\documentclass[11pt]{article}
\usepackage{fullpage}
\usepackage{fullpage, graphicx,psfrag}
\usepackage{epsfig}
\usepackage{color}
\usepackage{setspace}
\usepackage{amsmath,amsfonts,amssymb,amsthm}
\usepackage{latexsym,amssymb,amsmath,bm} 
\usepackage[]{graphicx} 
\usepackage{subfigure}
\usepackage[margin=1cm]{caption}
\usepackage{soul}
\usepackage{float}
\newtheorem{theorem}{Theorem}

\newtheorem{lemma}{Lemma}

\newtheorem{corollary}{Corollary}
\newtheorem{remark}{Remark}
\newcommand{\qedSolid}{\nobreak \ifvmode \relax \else
      \ifdim\lastskip<1.5em \hskip-\lastskip
      \hskip1.5em plus0em minus0.5em \fi \nobreak
      \vrule height0.75em width0.5em depth0.25em\fi}

\newcommand{\cd}{\overset{d}{\rightarrow}}    
\newcommand{\cp}{\overset{p}{\rightarrow}}    
\newcommand{\sumk}{\sum_{i=1}^{k}}

\newcommand{\abs}[1]{\left|{#1}\right|}

\newcommand{\rcov}{{\rm Cov}}

\newcommand{\maxp}{\max_{1\leq i \leq k} p_{1i}p_{2i}}

\newcommand{\psum}{||{\bf P}_1||_1^2 +|| {\bf P}_2||^2_2}
\newcommand{\vari}{||{\bf P}_1 + {\bf P}_2||^2_2}
\newcommand{\var}[1]{||{\bf P}_1 + {\bf P}_2||_2^{#1}}
\makeatletter

\begin{document}

\title{Two-Sample Test for Sparse High Dimensional Multinomial Distributions 
}

\date{}

\author{Amanda Plunkett  \thanks{Department of Defense, 9800 Savage Road, Ft. Meade, MD 20755, U.S.A.
\texttt{amanplunkett@gmail.com}}
and Junyong Park \thanks{Department of Mathematics and Statistics, University of Maryland Baltimore County, 1000 Hilltop Circle, Baltimore, MD 21250,  U.S.A.\texttt{junpark@umbc.edu}}
}


\maketitle

\begin{abstract}
In this paper we consider testing the equality of probability vectors of two independent multinomial distributions in high dimension.
The classical chi-square test may have some drawbacks in this case since many of cell counts may be zero or may not be large enough.
We propose a new test and show its asymptotic normality and
the asymptotic power function.
Based on the asymptotic power function, we present an application of our result to neighborhood type test
which has been previously studied, especially for the case of fairly small $p$-values.
To compare the proposed test with existing tests, we provide numerical studies including simulations and real data examples.

\end{abstract}

\section{Introduction}
\label{intro} 
In this paper, we discuss the problem of testing two multinomial distributions when
the number of categories is large.
Specifically, when we have two vectors ${\bf N}_c = (N_{c1}, \ldots, N_{ck})$ for $c=1,2$
   which follow multinomial distributions,  $Multinomial (n_c,  {\bf P}_c, k)$
         where  ${\bf P}_c =(p_{c1},p_{c2},\ldots, p_{ck})$ is a probability vector,
our testing scenario is
\begin{eqnarray} \label{objective}
H_0 : {\bf P}_1 = {\bf P}_2~~vs.~~H_1 : {\bf P}_1 \neq {\bf P}_2.
\end{eqnarray}

Our particular interest is a high dimensional multinomial with sparsity in the sense that
$k$ is large with a majority of categories having fairly small counts, such as  0, 1, or 2.
Typical examples are the cases where $k>n_c$  and  $ p_{ci}$'s are close to 0.
Some existing tests such as Pearson chi-square test
are based on large number of counts in each cell, however this may not occur under sparse data especially when $k$ is larger than $n_c$ for $c=1,2$.
The test that we propose is applicable to this sparse case and also more general cases including non-sparseness under some regular conditions presented later.

In fact, the hypothesis in (\ref{objective}) is equivalent to testing the equality of two mean vectors
of two multinomial distributions  $Multinomial (1,  {\bf P}_c, k)$  for $c=1,2$ with sample sizes $n_1$ and $n_2$.
For testing the equality of two population mean vectors, there are numerous studies.  
For example, see Bai and Saranadasa(1996), Chen and Qin (2010),
Srivastava (2009), Srivastava et al. (2013) and Park and Ayyala (2013).
However, multinomial distribution does not satisfy the assumptions such as factor models used in these references.

 On the other hand, Zelterman (1987) discussed goodness of fit tests in sparse contingency tables and also proposed the test
when the null probabilities are unknown. Zelterman (1987) includes the mean and variance of his proposed test
and proposed the normal approximation of standardized form of the test. From a theoretical point of view, Zelterman's test
requests some conditions on the cell probabilities and some relationship between
the number of cells and the frequency totals in contingency table.

It is worth while to noting that the goodness of fit test from one sample has a different context from the two sample problem.
In other words, the goodness of fit test is testing  $H_0 : {\bf P}= {\bf P}_0$ for a given ${\bf P}_0=(p_{01}, \ldots, p_{0k})$
and   ${\bf N}=(N_1, \ldots, N_k)$.
There are extensive studies on the goodness of fit testing problem for one sample such a s
Morris (1975), Cressie and Read (1984) and Kim et al. (2009) and
all these studies on goodness of fit tests are different from
the two sample problem in (\ref{objective}) in the sense that
test statistics for goodness of fit under the null hypothesis ${\bf P}={\bf P}_0$
utilize ${\bf P}_0$.
%

In this paper, we propose a new test statistic to test (\ref{objective}) for two samples
of  multinomial distributions. We provide asymptotic distribution and power function of the proposed test
and show numerical studies. In particular, we emphasize that our asymptotic results
provide more general results than Zelterman (1987).

This paper is organized as follows. In Section~\ref{Existing} we discuss existing methods that can be applied to our testing (\ref{objective}). In Sections~\ref{MultinomSetup}-\ref{ProofOfNormality}, we present our proposed test statistics and prove their asymptotic normality. We propose a new test statistic and show
its asymptotic null distribution and asymptotic power function.
In Section~\ref{NeighborhoodTest},
we consider an application of our proposed test based on asymptotic power function.
We define a neighborhood test, which is used in conjunction with our test statistic in Section~\ref{newsGroups} to analyze the 20 newsgroups dataset. In Section~\ref{MultinomSim} we show the performance of our test compared to other existing tests through the use of simulation experiments.  Concluding remarks are presented in section~\ref{MultinomialConclusion}.

\section{Existing methods for Comparison of Two Multinomial Distributions}
\label{Existing}
Suppose we have  ${\bf N}_c=(N_{c1}, N_{c2},\ldots, N_{ck})$ for $c=1,2$
which has the multinomial distribution, namely $Multinomial(n_c, {\bf P}_c, k) \equiv  {\cal M}(n_c, {\bf P}_c, k)$
where  $n_c = \sum_{i=1}^k N_{ci}$.
One typical  method for testing (\ref{objective}) is to use Pearson's $\chi^2$ test, which is
reliable when sample size in each cell
is large enough.
Pearson's $\chi^2$ statistic is defined as follows:
\begin{eqnarray}
\chi^2 = \sum_{c=1}^2 \sum_{i \in \{ i:N_{ci}>0 \}}  \frac{(N_{ci} - \hat{N}_{ci})^2 }{\hat{N}_{ci}}
\label{eqn:chi}
\end{eqnarray}
where $\hat{N}_{ci}=\hat{p}_{i}n_c$ for
$\hat p_i = \frac{N_{1i} + N_{2i}}{n_1 + n_2}$ is the expected count and $N_{ci}$ is the observed count
for the $i^{th}$ vector entry of the $c^{th}$ group.
As a related work, Anderson et al.(1972) applied a union-intersection method to
develop a procedure for testing the homogeneity of two sample multinomial data
and showed that their test is eventually equivalent to the Pearson chi-square test.
 The approximation based on chi-square distribution to (\ref{eqn:chi})
may be poor when the number of frequencies  $N_{ci}$
is not large enough. 

Alternatively, Zelterman (1987) proposed a goodness-of-fit statistic for contingency tables
 which provides improved power over the $\chi^2$ test when the $\chi^2$ is biased due to sparseness.
 They presented the conditional mean and variance of their proposed test conditioning on the marginal totals. They applied the asymptotic normality of the normalized form of their proposed test, which is effective especially for sparse and large dimensional contingency tables.
Zelterman's test is
\begin{eqnarray}
Z = \frac{\hat {D}_Z^2 - E(\hat{D}_Z^2 )}{\sqrt{Var(\hat{D}_Z^2 )}}
\label{eqn:Zelterman}
\end{eqnarray}
where $\hat{D}_Z^2 = \sum_{c=1}^2 \sum_{i=1}^{k^\star}  \frac{(N_{ci} - \hat{N}_{ci})^2 - N_{ci}}{\hat{N}_{ci}}$,
$\hat{N}_{ci}=\hat{p}_{i}n_c$ for  $\hat p_i = \frac{N_{1i} + N_{2i}}{n_1 + n_2}$,
and $N_{ci}$ is the observed value for the $i^{th}$ entry of the $c^{th}$ group.
Zelterman (1987) presented $E(\hat D_Z^2)$ and $Var(\hat D_Z^2)$.
From a theoretical point of view, Zelterman (1987) mentioned that the asymptotic normality of $Z$ in (\ref{eqn:Zelterman})
can hold
when $n$ and $k$ have the same increasing rate  and  the cell probabilities have the rates
 between $\frac{M_1}{k}$ and  $\frac{M_2}{k}$ for some constants $M_1 < M_2$.
%
%
%
%
Theses imply $np_{ci} \geq \epsilon>0$ for some constant $\epsilon$ which means that 
the expected counts under the null hypothesis should be bounded away from 0.  
Our proposed test is motivated by the estimator of 
Euclidean distance between two probability vectors and demonstrate some advantage over the test in Zelterman (1987) in two sample case. 
This advantage can be understood through both theory and numerical studies as we will show.

Additionally, there are many studies for testing the equality of mean vectors under some models such as factor models, 
for example, see  Bai and Saranadasa (1996), Chen and Qin (2010), Park and Ayyala (2013) and Srivastava (2009). 
As mentioned in the introduction, the multinomial distribution does not satisfy the conditions in all these studies. 
However, our problem for two multinomial distributions ${\bf N}_1$ and ${\bf N}_2$  
 is considered as testing (\ref{objective}) when 
 there are  ${\bf N}_{cl}$ where ${\bf N}_{cl} \sim Multinomial(1,{\bf P}_c, k)$ for $l=1,\ldots, n_c$ and $c=1,2$. 
 This is actually the case of testing the equality of mean vectors of ${\bf N}_{1l}$ and  ${\bf N}_{2l}$  which is 
 $ H_0 : {\bf P}_1 = {\bf P}_2$ in (\ref{objective}). 
  The tests in Park and Ayyala (2013) and Srivastava (2009) are not well defined in our setting  
  due to zero values in many cells. 
  We will consider the test in Bai and Saranadasa (1996) in our numerical studies while the test in Chen and Qin (2010) 
  is not practical under our situation due to computational complexity when $n_c$s are thousands.

In the following section, we propose a new test and show its asymptotic normality and
the asymptotic power under some conditions.
We will also provide numerical studies comparing our proposed test with existing methods as well as a real data example.


\section{New Test Statistic for Comparison of Two Multinomial Distributions}
\label{MultinomSetup}

In this section, we propose a new test and derive the asymptotic power of the proposed test from
the asymptotic normality under some regularity conditions.
\subsection{The Proposed Test Statistic}
We present a new procedure for testing the hypotheses in (\ref{objective})  when
the dimension of the multinomial vector is large. Our main goal is to propose a new test and derive the asymptotic distribution and asymptotic power function of the proposed test.
The proposed test is based on an unbiased estimator of
Euclidean distance between ${\bf P}_1$ and ${\bf P}_2$:
$\sumk (p_{1i}-p_{2i})^2 = ||{\bf P}_1 - {\bf P}_2 ||_2^2$ where
$||{\bf x}||_2 = \sqrt{\sumk x_i^2}$ for ${\bf x} =(x_1,\ldots, x_k)$.
Before we construct our test statistic, we mention that
we reformulate the multinomial distributed vector $(N_{c1},\ldots, N_{ck})$
as  the conditional distribution of  $(X_{c1},\ldots, X_{ck})$
given the total sum $\sum_{i=1}^k X_{ci}$ where $X_{ci}$s come from an independent Poisson distribution with mean $\lambda_{ci}= n_c p_{ci}$, i.e.,
$(N_{c1},\ldots, N_{ck}) \overset{d}{\equiv} (X_{c1},\ldots, X_{ck})|\sum_{i=1}^k X_{ci}=n_c$ where
$ \overset{d}{\equiv}$ means the equivalence of two distributions.
Morris (1975) provided asymptotic results for the multinomial distribution using Poisson distributions conditioning on
the total sum.
We first propose our test statistic based on independent Poisson distributions $(X_{ci})_{1\leq i \leq k}$ and then
we provide asymptotic results conditioning on the total sums, $\sum_{i=1}^k X_{ci}=n_c$.
In the observational vector of independent Poisson variables, say $ {\bf X}_c=(X_{ci})_{1\leq i \leq k}$ with
$X_{ci} \sim Poisson (\lambda_{ci})$ for $\lambda_{ci}=n_c p_{ci}$,
we define
\begin{eqnarray}
||{\bf P}_1 - {\bf P}_2||_2^2 = \left\vert \left\vert  \frac{\boldsymbol{\lambda}_1}{n_1} - \frac{\boldsymbol{\lambda_2}}{n_2} \right\vert \right\vert_2^2
= \sum_{i=1}^k \left(\frac{\lambda_{1i}}{n_1} - \frac{\lambda_{2i}}{n_2} \right)^2.
\label{eqn:euclidean}
\end{eqnarray}
and, to obtain an unbiased estimator for (\ref{eqn:euclidean}), we introduce
\begin{eqnarray}
f^*(x_1, x_2) =  \left(\frac{x_{1}}{n_1} -\frac{x_{2}}{n_2} \right)^2 - \frac{x_{1}}{n_1^2} - \frac{x_{2}}{n_2^2}.
\label{eqn:fstar}
\end{eqnarray}
We obtain an unbiased estimator of  $||{\bf P}_1 - {\bf P}_2 ||_2^2$  based on ${\bf X}_c$ for $c=1, 2$ which is
\begin{eqnarray}
{\cal D} \equiv \sum_{i=1}^k \left(\left(\frac{X_{1i}}{n_1} -\frac{X_{2i}}{n_2} \right)^2 - \frac{X_{1i}}{n_1^2} - \frac{X_{2i}}{n_2^2} \right) = \sum_{i=1}^k f^*(X_{1i},X_{2i})
\label{eqn:D}
\end{eqnarray}
satisfying  $E({\cal D}) =||{\bf P}_1 - {\bf P}_2||_2^2$.
Theorem \ref{thm:asymptoticdist} and Corollary \ref{cor:power} will show that the normalized form  $ \frac{{\cal D}}{\sqrt{\widehat {Var({\cal D})}}}$ for some estimator $\widehat{Var({\cal D})}$
has the asymptotic normal distribution for multinomial vector.
%
The Euclidean distance is commonly used for testing the equality of mean vectors of multivariate normal distributions
or factor models with some moment conditions.
See Bai and Saranadasa (1995) and Chen and Qin (2010).
In the context of testing in contingency tables,
the idea for the chi-square distribution is to consider
the goodness of fit for each cell using standardized quantities  under the null hypothesis,
$  \frac{(N_{ci} - n_c \hat p_i)^2}{ n_c \hat p_i}$ for
$\hat p_i = \frac{  N_{1i} + N_{2i}}{n_1 + n_2}$.
However, the denominator   $n_c \hat p_i$  in  (\ref{eqn:chi})
   is affected by cell probabilities which may lead to very skewed distribution for
small $p_i$s.  In our context, the sparse multinomial data are from small probabilities
in most of cells, so chi-square approximation to each cell may not be desirable.
On the other hand, our proposed tests based on ${\cal D}$ in (\ref{eqn:D}) first aggregate
estimates of $(p_{1i}-p_{2i})^2$  and then consider the normalization of ${\cal D}$.
This difference will lead to different performance between our proposed test and the test (\ref{eqn:Zelterman}).

We first present the following theorem which plays a major role in deriving the asymptotic distribution of our proposed test and
the asymptotic power.
We use the following notation: let $A=(a_{ij})_{1\leq i\leq m, 1\leq j \leq n}$ and  $|| A||_q = \left(\sum_{i,j} |a_{ij}|^q \right)^{1/q}$ for $q>0$.
Let ${\bf P}_c = (p_{c1},\ldots, p_{ck})$ for $c=1,2$
and  $\boldsymbol{\xi} ={\bf P}_{1}-{\bf P}_{2}$.
 For two vectors ${\bf P}_1$ and ${\bf P}_2$, the dot product is
${\bf P}_1 \cdot {\bf P}_2 = \sum_{i=1}^k p_{1i}p_{2i}$ and
component-wise product of ${\bf P}_1$ and ${\bf P}_2$ is ${\bf P}_1*{\bf P}_2 = (p_{11}p_{21}, \ldots, p_{1k}p_{2k})$.
We also define $\sqrt{{\bf P}_c} = (\sqrt{p_{c1}},\ldots, \sqrt{p_{ck}})$ and  $|\boldsymbol{\xi}| =(|\xi_1|,\ldots, |\xi_k|)$.
Let $n$ be a sequence satisfying $ n_1 \asymp n_2 \asymp n$
where $A_n \asymp B_n $  implies   $0< \lim\inf_n \frac{A_n}{B_n} \leq  \lim\sup_n \frac{A_n}{B_n} <\infty$
for sequences $A_n>0$ and $B_n>0$.
The notion $\cd$ implies convergence in distribution.

\begin{theorem}
Let ${\bf N}_c$ be independent multinomial random vectors  for $c=1,2$ such as ${\cal M}(n_c, {\bf P}_c, k)$ where ${\bf N}_c=(N_{ci})_{1\leq i \leq k}$  and
${\bf P}_c=(p_{ci})_{1\leq i \leq k}$ for $c=1,2$.
Suppose the following conditions are satisfied: for $n=n_1 + n_2$,
\begin{eqnarray*}
&&\mbox{Condition 1:~~~~} \min (n_1, n_2) \rightarrow \infty,~~ \frac{n_1}{n} \rightarrow c \in (0,1), \\
&&\mbox{Condition 2:~~~~}   \frac{\max_i p_{ci}^2}{||{\bf P}_c ||_2^2}  \rightarrow 0 ~~ \mbox{for $c=1,2$ as $k \rightarrow \infty$,}  \label{uan_result}\\
&&\mbox{Condition 3:~~~~}   n \vari \geq \epsilon >0 ~~~\mbox{for some $\epsilon>0$,}     \\
&&\mbox{Condition 4:~~~~}   {n^2|| \boldsymbol{\xi} ||_2^4} = O(\vari   ).
\end{eqnarray*}
Then, we have
\begin{eqnarray}
\frac{ \sumk f^*(N_{1i}, N_{2i})-||{\boldsymbol \xi} ||_2^2}{\sigma_k}   \cd N(0,1)
\end{eqnarray}
where 
\begin{eqnarray}
\sigma_k^2 =   2 \sum_{i=1}^k \left(\frac{p_{1i}}{n_{1}} + \frac{p_{2i}}{n_2} \right)^2
\label{eqn:var}
\end{eqnarray}

  and $f^*$ is given by (\ref{eqn:fstar}).
\label{thm:asymptoticdist}
\end{theorem}
\begin{proof}The proof of Theorem \ref{thm:asymptoticdist} will be provided in section 4 with a series of lemmas.
\end{proof}

\begin{remark}
The conditions in Theorem \ref{thm:asymptoticdist}  will be used throughout this paper.
The sample sizes $n_c$ for $c=1,2$ and the dimension $k$.
do not have explicit relationship.
This is in contrast to
(\ref{eqn:Zelterman})  in Zelterman (1987)
  assuming that $k$ and $n$ have the same increasing rate for the theoretical proof of the asymptotic normality.  Instead, our conditions in Theorem \ref{thm:asymptoticdist}
  do not require direct relationship between $k$ and $n_c$.
Rather, the relationship between $k$ and $n$ are only through
  Conditions 3 in Theorem \ref{thm:asymptoticdist}.
For example, when $p_{ci} \asymp 1/k$ and $n \asymp k$, then the condition 3 requests
$k = O(n)$ which includes the case of $k \asymp n$ in Zelterman (1987).
However, the condition 3 covers a variety of situations compared to Zelterman (1987).
For example, when  $p_{ci} = \frac{1/{i}}{\sum_{i=1}^k 1/{i}} \sim  \frac{1}{{i} \log k}$ and $(\log k)^2 =O(n)$, 
all four conditions in Theorem \ref{Theorem1} are satisfied.  
The condition $(\log k)^2 =O(n)$ allows 
$k$ to increase at the rate of $\exp(\sqrt{n})$. 
In other words, our conditions include more general relationship between $n_c$ and $k$ through 
depending on the configurations of $p_{ci}$s. 
\end{remark}

In Theorem \ref{thm:asymptoticdist}$,
\sum_{i=1}^n f^*(N_{1i},N_{2i})$ is known, however $\sigma_k^2$ is unknown, so
we need to have some estimates of $\sigma_k^2$ defined in (\ref{eqn:var})
which have an asymptotically equivalent behavior.
Our proposed test is constructed under the null hypothesis $H_0$: ${\bf P}_1 = {\bf P}_2$.
For derivation of $\sigma_k^2$, see the proof of Lemma~\ref{lemma:variance} in section \ref{ProofOfNormality}.
In practice, we need some estimate of $\sigma_k^2$ based on multinomial data ${\bf N}_c$ for $c=1,2$.
We propose an estimator  of $\sigma_k^2$ which is 
\begin{eqnarray}
\hat \sigma_k^2 &=&  \sumk \sum_{c=1}^2\frac{2}{n_c^2} \left(\hat p_{ci}^2 - \frac{\hat p_{ci}}{n_c} \right)
+ \frac{4}{n_1n_2} \sumk \hat p_{1i} \hat p_{2i}
\label{eqn:sigma2}
\end{eqnarray}
where  $p_i=\frac{n_1 p_{1i} + n_2 p_{2i}}{n_1 + n_2}$ and $\hat p_{ci} =  \frac{N_{ci}}{n_c}$.  
Lemma~\ref{lemmaConvProb} states that the proposed estimator of $\sigma_k^2 $  has 
the property of ratio consistency.



\begin{lemma} \label{lemmaConvProb}
Under conditions 1 and 2  in Theorem \ref{thm:asymptoticdist},
$\frac{\hat \sigma_k^2}{\sigma_k^2} \cp 1$.
\label{lemma:ratioconsistency}
\end{lemma}
\begin{proof}
See Appendix. 
\end{proof}

Based on the estimators $\hat \sigma_k^2$, we define the following two test statistics, namely $T$;
\begin{eqnarray}
T \equiv \frac{\sumk f^*(N_{1i}, N_{2i})}{\hat \sigma_k}
\label{eqn:test}
\end{eqnarray}
where $f^*$ is defined in (\ref{eqn:fstar}).
From Theorem \ref{thm:asymptoticdist} and Lemma \ref{lemma:ratioconsistency},
$T$ is asymptotic normal  under the $H_0$.
We state this in the following corollary.
\begin{corollary} \label{Theorem1}
Under $H_0$, if Conditions 1 and 2 in  Theorem \ref{thm:asymptoticdist} are satisfied,
then  $T \cd  N(0,1)$ where $T$ is defined in (\ref{eqn:test}).
\end{corollary}

Corollary \ref{Theorem1} shows that
our proposed test is available for practical use under fairly mild conditions  $1-3$ of Theorem \ref{thm:asymptoticdist}.  
Based on Corollary \ref{Theorem1}, we reject $H_0$ if
\begin{eqnarray}
T > z_{1-\alpha}
\end{eqnarray}
where $z_{1-\alpha}$ is the $1-\alpha$ quantile of a standard normal distribution.
In practice, our test requests only conditions $1-3$ 
of Theorem \ref{thm:asymptoticdist} to have asymptotic size $\alpha$ test for a given $\alpha \in (0,1)$. 
Additionally,  it is of interest to investigate the power function of our proposed tests.
In particular, the power function from Theorem \ref{thm:asymptoticdist}
is meaningful  when the signal-to-noise ratio
$SNR \equiv  ||\boldsymbol{\xi} ||_2^2/\sigma_k$  for ${\boldsymbol \xi}={\bf P}_1 - {\bf P}_2$  is bounded,
i.e.,  $ SNR =O(1).$ which is the case that the asymptotic power is non-trivial in the sense that the power is in $(0,1)$.
Condition 4 in Theorem \ref{Theorem1} is equivalent to the condition that 
the SNR is bounded by some constant as $k\rightarrow \infty$.
\begin{corollary} Under the conditions in Theorem \ref{thm:asymptoticdist}, we have
\begin{eqnarray}
P\left(T> z_{1-\alpha} \right) - \bar \Phi \left(z_{1-\alpha} - \frac{||{\bf P}_1 - {\bf P}_2 ||_2^2}{\sigma_k}    \right) \label{eqn:power1}
\rightarrow 0
\end{eqnarray}
where $\bar \Phi (z)  = 1-\Phi(z) = P(Z>z)$ for a standard normal random variable $Z$ and $z_{1-\alpha}$ is  the $(1-\alpha)$ quantile of a standard normal distribution.
\label{cor:power}
\end{corollary}
\begin{proof}
From Theorem \ref{Theorem1} and Lemma \ref{lemma:ratioconsistency}, we have (\ref{eqn:power1}).
\end{proof}

\begin{remark}
It is clear that under $H_0$, the proposed test is asymptotically size-$\alpha$ test since $||{\bf P}_1 - {\bf P}_2||={\bf 0}$ under $H_0$.
\end{remark}

In the following section, we provide the proof of Theorem \ref{thm:asymptoticdist}.


\section{Asymptotic Normality of the proposed tests}
\label{ProofOfNormality}
In this section we prove Theorem~\ref{Theorem1}.
The main difficulty is the dependency imposed by the multinomial distribution. In other words,
$f^*(N_{1i}, N_{2i})$s are not independent since $N_{ci}$s have dependency for $1\leq i \leq k$ from the multinomial distributions.
Therefore, it is not straightforward to apply the central limit theorem based on the assumption of independence.
Instead,  
Steck (1957) and Morris (1975) use conditional central limit theory for independent Poisson distributions conditioning on sums of Poisson variables to have the asymptotic normality of multinomial distributions. 
More specifically, to avoid the issue of dependency from the multinomial distribution, we use the fact that
the multinomial random vector $(N_{c1}, N_{c2},\ldots, N_{ck})$ has the same distribution as $(X_{c1}, \ldots, X_{ck}) | \sum_{i=1}^k X_{ci}=n_c $ where $X_{ci}$s are independent Poisson random variables with mean $\lambda_{ci}=n_cp_{ci}$.
Before we present our main results,
we first define the following notations:
\begin{eqnarray}
f_i(x_{1i},x_{2i}) &=& \underbrace{f^*(x_{1i}, x_{2i}) 
 - (p_{1i}-p_{2i})^2}_{{\cal G}_{1i} (x_{1i}, x_{2i})} \\
&&  \underbrace{ - 2(p_{1i}-p_{2i}) \left(\frac{x_{1i}}{n_1} -\frac{x_{2i}}{n_2} \right)
+2(p_{1i}-p_{2i})^2}_{{\cal G}_{2i}(x_{1i}, x_{2i} ) } \label{eqn:fi} \\
F_k &=&   \frac{\sum_{i=1}^k f_i(X_{1i}, X_{2i})}{\sigma_k} \nonumber \\
    &=& \frac{ \sum_{i=1}^k {\cal G}_{1i}(X_{1i}, X_{2i})}{\sigma_k}
+ \frac{ \sum_{i=1}^k {\cal G}_{2i}(X_{1i}, X_{2i})}{\sigma_k} \label{eqn:F_k} \\
U_{ck} &=& \frac{1}{\sqrt{n_c}} \sum_{i=1}^k (X_{ci} -\lambda_{ci})~~~\mbox{for $c=1,2$.}
\end{eqnarray}
We will show that   $(i)$  $ (F_k, U_{1k}, U_{2k})' \cd N_3 ( (0,0,0)', I_3 )$ which is a trivariate multinormal distribution
where $I_3$ is a  $3 \times 3$ identity matrix
and $(ii)$   $(F_k| U_{1k}=0, U_{2k}=0)  \cd N(0,1)$.
The latter case means that, under the condition of $U_{ck} =0$ (equivalently  $\sum_{i=1}^k X_{ci}=n_c$) for $c=1,2$,
 the conditional distribution of $ F_k$ is the same as that of $\frac{\sum_{i=1}^k f^*(N_{1i},N_{2i}) - ||\boldsymbol{\xi} ||_2^2}{\sigma_k} +\frac{\sum_{i=1}^k {\cal G}_{2i} (N_{1i},N_{2i})}{\sigma_k}$ since 
  $ [(X_{c1},\ldots, X_{ck})| U_{ck}=0 ]   \overset{d}{\equiv} (N_{c1},\ldots,N_{ck})$.
For the asymptotic normality of $(F_k| U_{1k}=0, U_{2k}=0)$,
we need the uniform equicontinuity for the conditional central limit theorem as stated in Theorem 2.1 in Steck (1957).
For the uniform equicontinuity in Steck (1957),
we need to show that, for bounded values $|u_1|\leq \delta$ and $|u_2|\leq \delta$ for some $\delta>0$
and $h=\max(h_1,h_2)$,
the conditional characteristic function of $F_k$ given $U_{1k}=u_1$ and $U_{2k}=u_2$ satisfies
\begin{eqnarray*}
 &&\lim_{h\rightarrow 0} \sup_k \sup_{|u_1|\leq \delta, |u_2|\leq \delta} |E (e^{i t F_k} | U_{k1}=u_1+ h, U_{k2}=u_2+ h) - E (e^{it F} | U_{k1}=u_1, U_{k2}=u_2) |\\
 && \rightarrow 0. 
 \end{eqnarray*}
We will show the uniform equicontinuity of the characteristic function in  Lemma \ref{lemmaMorrisCondition}. 

%
From Theorem 2.1 in Steck (1957),
the uniform equicontinuity of characteristic function of $F_k$
implies the conditional asymptotic normality of $F_k$ given $U_{1k}=U_{2k}=0$, i.e.,  $F_k|U_{1k}=0, U_{2k}=0 \cd N(0,1)$.

The following Lemmas, \ref{lemma:variance} and \ref{lemmaMorrisCondition}, 
will be used in showing the asymptotic multivariate normality of $(F_k, U_{1k}, U_{2k})$ and the uniform equicontinuity of the characteristic function of $F_k$ conditioning on $U_{1k}=0$ and $U_{2k}=0$.

In fact, the uniform equicontinuity of characteristic function becomes   
\begin{eqnarray*}
 &&\lim_{h\rightarrow 0} \sup_k \sup_{|u_1|\leq \delta, |u_2|\leq \delta} |E (e^{i t F_k} | U_{k1}=u_1+ h, U_{k2}=u_2+ h) - E (e^{it F} | U_{k1}=u_1, U_{k2}=u_2) |\\
 &\leq& \lim_{h\rightarrow 0} \sup_k \sup_{|u_1|\leq \delta, |u_2|\leq \delta} E \exp \left( i \frac{t}{\sigma_k} \sumk (f_i( L_{1i} + M_{1i}, L_{2i} + M_{2i}) - f_i(L_{1i}, L_{2i})) \right)\\
 &\leq& \lim_{h\rightarrow 0} \sup_k \sup_{|u_1|\leq \delta, |u_2|\leq \delta}  \frac{|t|}{\sigma_k}  E \Bigg|\sumk  (f_i(L_{1i}+M_{1i}, L_{2i}+M_{2i}) -f_i(L_{1i}, L_{2i})) \Bigg| \\
 &\leq&  \lim_{h\rightarrow 0} \sup_k \sup_{|u_1|\leq \delta, |u_2|\leq \delta} \left( \frac{t^2}{\sigma^2_k}  E \left(\sumk  (f_i(L_{1i}+M_{1i}, L_{2i}+M_{2i}) -f_i(L_{1i}, L_{2i}))\right)^2    \right)^{1/2} 
 \end{eqnarray*}
and it is sufficient to show that the last expression converges to 0. 

\begin{lemma}
When $X_{1i}$ and $X_{2i}$ for $1 \le i \le k$ are independent Poisson random variables with means $\lambda_{1i}=n_1p_{1i}$ and $\lambda_{2i}=n_2 p_{2i}$, respectively,
then
\begin{enumerate}
\item  $Ef_i(X_{1i},X_{2i})=0$.
\item  $Cov(\sumk f_i(X_{1i}, X_{2i}) , \sumk X_{ci})=0$ for $c=1,2$.
\item $Var(\sum_{i=1}^k f_i(X_{1i},X_{2i})) =  \sum_{i=1}^k Var(f_i(X_{1i}, X_{2i}))
=  2\sumk \left(\frac{p_{1i}}{n_1} + \frac{p_{2i}}{n_2} \right)^2$.
\end{enumerate}
\label{lemma:variance}
\end{lemma}
\begin{proof} See Supplementary material.
\end{proof}



The following lemma ensures that the convergence of characteristic function of $\sum_{i=1}^k f(X_{1i}, X_{2i})$  based on independent Poisson distributions conditioning on $\sum_{i=1}^k X_{1i}$ and $\sum_{i=1}^k X_{2i}$ which come from multinomial distributions.
\begin{lemma} \label{lemmaMorrisCondition}
When  ${\bf L}_{ck}=(L_{c1},...,L_{ck})$ and ${\bf M}_{ck}=(M_{c1},...,M_{ck})$ are
independent multinomial vectors  for $c=1,2$ with ${\cal L}({\bf L}_{ck}) = {\cal M}({\bf P}_c, k, n_c+ u_c n_c^{1/2})$ and ${\cal L}({\bf M}_{ck})={\cal M}({\bf P}_c, k, h_c n_c^{1/2})$.
$h_c$ and $u_c$ are such that $l_c = n_c + u_c n_c^{1/2}$ and $m_c = h_c n_c^{1/2}$ are nonnegative integers and $u_c$ is bounded
 (say $|u_c|\leq \delta$ for some constant $\delta$, $c=1,2$)as $k \rightarrow \infty$.
Under the conditions in  Theorem \ref{thm:asymptoticdist},
for $h=\max(h_1, h_2)$, we have 
\begin{align}
	\lim_{h \rightarrow 0} \sup_k \sup_{|u_1|\leq \delta, |u_2|\leq \delta}\frac{1}{\sigma_k^2}E \left[ \left( \sum_{i=1}^k   f_i(L_{1i}+M_{1i},L_{2i}+M_{2i}) - f_i(L_{1i},L_{2i})  \right)^2  \right] = 0. \label{eqn:MorrisCondition}
\end{align}
\end{lemma}
\begin{proof} See supplementary material.  \end{proof}

The following lemma shows that $\sumk f(X_{1i},X_{2i})$ has the asymptotic normality 
when $X_{ci}$s are independent poisson distributions.    

\begin{lemma} 
When $X_{1i}$ and $X_{2i}$ for $1 \le i \le k$ are independent Poisson random variables with means $\lambda_{1i}=n_1p_{1i}$ and $\lambda_{2i}=n_2 p_{2i}$, respectively,
then  
\begin{eqnarray}
\frac{\sum_{i=1}^k f(X_{1i},X_{2i})}{\sigma_k} \cd N(0,1). 
\end{eqnarray}
\label{lemma:normal_poisson}
\end{lemma}
\begin{proof} 
See the Supplementary material. 
\end{proof}

Based on the lemmas, we prove Theorem \ref{Theorem1}.
In fact, Theorem \ref{Theorem1} is the case  
when independent poisson random variables $X_{ci}$s in Lemma \ref{lemma:normal_poisson} can be replaced by 
the multinomial distributions $N_{ci}$s.

\bigskip
\noindent {\bf Proof of Theorem \ref{Theorem1}} :
%
%
Lemma \ref{lemma:normal_poisson} shows  $F_k = \frac{\sum_{i=1}^k f_i(X_{1i},X_{2i})}{\sigma_k}  \cd N(0,1)$.
We also have $U_{ck} = \frac{1}{\sqrt{n_c}} \sumk (X_{ci} -\lambda_{ci})  \cd N(0,1)$ for $c=1,2$ from 
the Lyapounov' condition : $ \frac{\sum_{i=1}^k E(X_{ci} -\lambda_{ci})^4}{(\sum_{i=1}^k \lambda_i^2)} 
=  \frac{ 3\sum_{i=1}^k \lambda_i^2 + \sum_{i=1}^k \lambda_i}{(\sum_{i=1}^k \lambda_i^2)}  = \frac{3}{n^2 \sum_{i=1}^k p_{ci}^2} 
+ \frac{3}{n^3 (\sum_{p_{ci}^2})} \rightarrow 0 $ from the condition 3 in Theorem \ref{Theorem1}.   
Using Lemma \ref{lemma:variance} and independence of ${\bf X}_1$ and ${\bf X}_2$,
we have the result that  $ \frac{\sumk f(X_{1i},X_{2i})}{\sigma_k}$, $U_{1k}$ and $U_{2k}$ are uncorrelated to each other. Therefore,
 using Lemma 2.1 in Morris (1975), we have tri-variate asymptotic normality of $(F_k, U_{1k}, U_{2k})$, i.e.,
 $(F_k, U_{1k}, U_{2k})' \cd  N_3( (0,0,0)',  I_3)$ where $I_3$ is a $3\times 3$ identity matrix.
Lemma \ref{lemmaMorrisCondition} shows the uniform equicontinuity of conditional characteristic function of $F_k$ given $U_{1k}$ and $U_{2k}$, so we have
$ (F_k|U_{1k}=U_{2k}=0) \cd  N(0,1)$, in other words 
\begin{eqnarray}
(F_k|U_{1k}=U_{2k}=0) \overset{d}{\equiv} \frac{\sumk f_i(N_{1i},N_{2i})}{\sigma_k} 
\cd  N(0,1). 
\label{eqn:Fk_normal} 
\end{eqnarray}

From (\ref{eqn:F_k}), conditioning on $U_{1i}=U_{2k}=0$,   
we have 
$ \frac{\sumk f^*(N_{1i},N_{2i}) -||\boldsymbol{\xi}||_2^2} {\sigma_k}  
= \frac{\sum_{i=1}^k {\cal G}_{1i}(N_{1i},N_{2i})}{\sigma_k} = \frac{\sumk f_i(N_{1i},N_{2i}) }{\sigma_k}    - \frac{\sum_{i=1}^k {\cal G}_{2i}(N_{1i},N_{2i})}{\sigma_k}.$ 
From (\ref{eqn:Fk_normal}), we only  to show
$\frac{\sum_{i=1}^k {\cal G}_{2i}(N_{1i},N_{2i})}{\sigma_k} \cp 0$ to have the asymptotic normality of  $\frac{\sum_{i=1}^k {\cal G}_{1i}(N_{1i},N_{2i})}{\sigma_k}$.
For this, it is enough to show $Var( \sum_{i=1}^k {\cal G}_{2i}(N_{1i},N_{2i}))=o(\sigma_k^2)$
since  $ E \left(\sum_{i=1}^k  {\cal G}_{2i}(N_{1i},N_{2i}) \right)  =0$. 
Using $Var(N_{ci})=n_c p_{ci}(1-p_{ci})$ and $Cov(N_{ci}, N_{ci}) = - n_c p_{ci}p_{cj}$ for $c=1,2$, we have
\begin{eqnarray*}
Var (\sum_{i=1}^k {\cal G}_{2i}(N_{1i}, N_{2i}) ) 
&=& \sum_{i=1}^k \xi_i^2 \left(\frac{p_{1i}(1-p_{1i})}{n_1} +  \frac{p_{2i}(1-p_{2i})}{n_2}\right)
- \sum_{i \neq j} \xi_i \xi_j \left( \frac{p_{1i}p_{1j}}{n_1} + \frac{p_{2i}p_{2j}}{n_2}  \right)\\
&=& \sum_{i=1}^k \xi_i^2 \left(\frac{p_{1i}}{n_1} +  \frac{p_{2i}}{n_2}\right) -\sum_{c=1}^2(\sum_{i=1}^k \xi_i p_{ci})^2 
\leq  \sum_{i=1}^k \xi_i^2 \left(\frac{p_{1i}}{n_1} +  \frac{p_{2i}}{n_2}\right)
\end{eqnarray*}
where the last equality is due to   $ (\sum_{i\neq j} \xi_i \xi_j   p_{ci} p_{cj})  =  (\sum_{i=1}^{k} \xi_i  p_{ci})^2 - \sum_{i=1}^k \xi_i^2 p_{ci}^2$ for $c=1,2$. 
Since $\sum_{i=1}^k \xi_i^2 \left(\frac{p_{1i}}{n_1} +  \frac{p_{2i}}{n_2}\right)  
 \leq  O\left( \frac{1}{n} ||\boldsymbol{\xi} ||_2^2 \right)   (\max_i p_{1i} + \max_i p_{2i}) 
 \leq    \frac{\var{}}{n^2}  \var{} =  o(\sigma_k^2)$
  where the last equation is from  the condition 4 in Theorem \ref{Theorem1}, $ \max_i p_{ci} =o(\var{})$ 
  and $\sigma_k^2 \asymp n^{-2} \vari$.    
Therefore, using  (\ref{eqn:Fk_normal}) and  $Var( \sum_{i=1}^k {\cal G}_{2i}(N_{1i},N_{2i}))=o(\sigma_k^2)$,  we have  
\begin{eqnarray*}
   \frac{ \sum_{i=1}^k f^*(N_{1i}, N_{2i}) -||\boldsymbol{\xi}||^2  }{\sigma_k} =  
\frac{\sumk f_i(N_{1i}, N_{2i})}{\sigma_k} - \frac{\sumk {\cal G}_{2i}(N_{1i}, N_{2i})}{\sigma_k}    \cd N(0,1)
\end{eqnarray*}
\qed


\section{Neighborhood Test}
\label{NeighborhoodTest}
{ In Corollary 2, we presented the closed form asymptotic power of the proposed test.  From the closed form of asymptotic power in Corollary \ref{Theorem1},  we may expect additional applications.
In this section, we present one application based on the closed form of asymptotic power of $T$ in Corollary \ref{Theorem1}.}

In testing the equality of parameters from two populations,
it frequently happens
that  the null hypothesis is rejected even though the estimates of effect sizes are close to
each other, however, these differences are so small that parameters may not be
considered to be different in practice.
Another issue is that although the use of $p$-values is a common measure to draw a conclusion about
the population, one may be interested in the measure of indifference or inhomogeneity
regarding the original effect sizes based on ${\bf P}_1$ and ${\bf P}_2$.
As related work, see Solo (1984), Berger and Delampady (1987), Berger and Sellke (1987),
Dette and Munk (1998), Munk et al. (2008) and Choi and Park (2014).
In particular, Munk et al. (2008) called  this type of testing problem a neighborhood test.
With these motivations, instead of testing the exact equality such as  $H_0 : {\bf P}_1 = {\bf P}_2$,
we consider more flexible null hypothesis,
which allows a predetermined level of difference such as
${\cal N}_{\delta} = \{ ({\bf P}_1,{\bf P}_2) : d({\bf P}_1, {\bf P}_2) \leq \delta \}$.
Here,  $d$ is a function satisfying $d({\bf P}_1, {\bf P}_2)=0$  under ${\bf P}_1 = {\bf P}_2$.
In general, when considering  ${\cal N}_{\delta}$ as a null space for equivalence of ${\bf P}_1$ and ${\bf P}_2$,
there is an important issue in the determination of
the rejection region for a given neighborhood to have a size $\alpha$ test for a given $\alpha$.
That is, for a given test  ${T}$,  we need to find out $C$ satisfying
\begin{eqnarray}
\sup_{({\bf P}_1, {\bf P}_2) \in {\cal N}_{\delta}} P_{ ({\bf P}_1, {\bf P}_2) } ({T} > C ) =\alpha.
\end{eqnarray}
Choi and Park (2014) discussed testing non-equivalence of normal mean values and
found the least favorable parameters for different types of null hypotheses.
Munk et al. (2008) considered a noncentral chi-square distribution in a neighborhood test for functional data analysis.
In our case, we consider a testing problem based on SNR (signal to noise ratio) which influences the effect size in the two sample test as follows:
\begin{eqnarray}
{\cal N}_{\delta} = \left\{ ({\bf P}_1, {\bf P}_2) :  \frac{||{\bf P}_1 - {\bf P}_2||^2_2}{\sigma_k} \leq \delta \right\}.
\label{eqn:neighborhood2}
\end{eqnarray}
Note that $\delta=0$ implies $H_0 : {\bf P}_1 = {\bf P}_2$.
We test
\begin{eqnarray}
H_{0,\delta} : ({\bf P}_1, {\bf P}_2) \in {\cal N}_{\delta}~~vs.~~H_1 : ({\bf P}_1, {\bf P}_2) \notin {\cal N}_{\delta}
\label{eqn:neighborhood}
\end{eqnarray}
When $N_{\delta}$ in (\ref{eqn:neighborhood2}) is given,
the power function of $T$ in Corollary \ref{cor:power} gives
some insight into the rejection region for a given size $\alpha$.
For a given $\alpha$, the goal is to identify $C$ satisfying
\begin{eqnarray}
\lim_{n_1, n_2 \rightarrow \infty} \sup_{({\bf P}_1, {\bf P}_2) \in {\cal N}_{\delta}} P_{({\bf P}_1, {\bf P}_2)}\left( T > C \right) =\alpha
\end{eqnarray}
for $T$.
The supremum occurs when $C=z_{1-\alpha}-\delta$ and  $\frac{||{\bf P}_1 - {\bf P}_2||^2_2}{\sigma_k}=\delta$
from the asymptotic power function of $T$. The asymptotic $p$-value is
\begin{eqnarray*}
p_{\delta} =\bar \Phi \left( \frac{D}{\hat \sigma_k} -\delta \right)
\end{eqnarray*}
where $\bar \Phi (x) = P(Z>z)$ and $Z$ has a standard normal distribution.
Since ${p}_{\delta}$ is a monotone increasing function of $\delta$,
we have $p_{\delta} \rightarrow 1$ as $\delta$ increases.
When the $p$-value from testing $H_0 : {\bf P}_1 = {\bf P}_2$ is almost 0,
we can obtain some $\delta^*(\alpha)$ for a given $\alpha$ satisfying
\begin{eqnarray}
\delta^*(\alpha) =  \min_{\delta>0} \{ \delta :  p_{\delta} \geq \alpha \}.
\end{eqnarray}
In Munk et al. (2008),  $\delta^*(\alpha)$ is called the size of the test for a given $\alpha$ and
can be presented as a measure of indifference of ${\bf P}_1$ and ${\bf P}_2$ instead of a $p$-value from testing $H_0 : {\bf P}_1 = {\bf P}_2$.
Park et al. (2015) and Choi and Park (2014) investigated the behavior of $\delta^*(\alpha)$
for different problems of testing normal means.

We apply this neighborhood test to a real data example in section \ref{newsGroups}.
\section{Simulations}
\label{MultinomSim}
In this section, we provide numerical studies to compare the proposed test ($T$) with existing tests such as
the test  in (\ref{eqn:Zelterman}) and the test (BS-test) in Bai and Saranadasa (1996).

Throughout all following simulations,  we repeat $10^4$ simulations to compute each of empirical sizes or powers. 
We first investigate the sizes of three tests when $k$ is larger than sample sizes.
We consider two types of scenario: $(i)$  $k$  increases  when the ratio of the dimension and sample sizes is 10, i.e., $k/n_c=10$.
In these cases, the sample sizes also increase as $k$ increases.
As the configurations of ${\bf P}_1={\bf P}_2$, we use two cases:
$(i)$  $p_{1i}=p_{2i}= \frac{\frac{1}{i^{\gamma}}}{\sumk  \frac{1}{i^{\gamma}}}$  for $\gamma=0.45$ 
$(ii)$ $p_{1i}=p_{2i}=\frac{1}{k}$.
$(ii)$ $k$ increase when sample sizes are fixed such as $n_1=n_2=10^3$. In these cases,
data are getting more sparse as $k$ increases.

Tables \ref{tab:sizes1} and \ref{tab:sizes2} show
$(i)$ and $(ii)$, respectively. As displayed in Tables \ref{tab:sizes1} and \ref{tab:sizes2}, we see that
the proposed test($T$) and Zelterman's test control the nominal level of size (0.05) reasonably, however
BS-test fails in controlling the nominal level since the BS-test always achieves inflated sizes up to $10\%$.

\begin{table}[ht]
\centering
\begin{tabular}{rrrr|rrrr}
\hline
\multicolumn{4}{c}{$p_{1i}=p_{2i} =  1/i^{0.45} /\sum_{i=1}^k 1/i^{0.45}   $} & \multicolumn{4}{c}{$p_{1i}=p_{2i}=1/k$} \\
  \hline
 $k$ & $T$ & $BS$ & $Zel$  &   $k$ & $T$ & $BS$ & $Zel$ \\
  \hline
$10^3$       & 0.051& 0.100& 0.052 &$10^3$       & 0.054 &0.116& 0.060\\
$10^4$       & 0.065& 0.100& 0.064 &$10^4$       & 0.046 &0.086& 0.051\\
$2\times10^4$& 0.051& 0.076& 0.048 &$2\times10^4$& 0.049 &0.084& 0.053\\
$3\times10^4$& 0.057& 0.085& 0.057 &$3\times10^4$& 0.038 &0.064& 0.040\\
$10^5$ & 0.058& 0.086& 0.046 &$10^5$ & 0.052 &0.089& 0.052\\
   \hline
\end{tabular}
\caption{Empirical sizes of tests when the nominal level is $0.05$ and $k/n_c =10$ for $c=1,2$. }
\label{tab:sizes1}
\end{table}

\begin{table}[ht]
\centering
\begin{tabular}{rrrr|rrrr}
\hline
\multicolumn{4}{c}{$p_{1i}=p_{2i} =  1/i^{0.45} /\sum_{i=1}^k 1/i^{0.45}   $} & \multicolumn{4}{c}{$p_{1i}=p_{2i}=1/k$} \\
  \hline
 $k$ & $T$ & $BS$ & $Zel$  &   $k$ & $T$ & $BS$ & $Zel$ \\
  \hline
$10^3$       & 0.041&  0.079&  0.052&  $10^3$       & 0.049&  0.079&  0.050\\
$10^4$       & 0.052&  0.094&  0.053&  $10^4$       & 0.046&  0.075&  0.046\\
$2\times10^4$& 0.059&  0.103&  0.062&  $2\times10^4$& 0.064&  0.102&  0.059\\
$3\times10^4$& 0.050&  0.093&  0.049&  $3\times10^4$& 0.056&  0.092&  0.057\\
$10^5$ & 0.037&  0.093&  0.038&  $10^5$ & 0.040&  0.071&  0.038\\
   \hline
\end{tabular}
\caption{Empirical sizes of tests when the nominal level is $0.05$ and sample sizes are fixed, $n_1=n_2=10^3$.  }
\label{tab:sizes2}
\end{table}

We now consider powers of three tests.
Our simulation set up is as follows.
\begin{itemize}
\item \textbf{Experiment 1}: $p_{1i}=\frac{1/i^\gamma}{\sum_{i=1}^k 1/i^\gamma}$ for $\gamma=0.45$. The probability vector for the $2^{nd}$ group was generated by switching the position of  1st and $m$th entries, i.e., $p_{2,1}=p_{1,m}$, $p_{2,m}=p_{1,1}$ and $p_{1i}=p_{2i}$
    for  all $i\neq 1, m$.
\item \textbf{Experiment 2}: $p_{1i}=1/k$ for $1\leq i\leq k$, $p_{2i}=0$  for $i \in [1,b]$, $p_{2,b+1}=\sum_{i=1}^{b+1}p_{1i}=\frac{b+1}{k}$, $p_{2i}=1/k$ for $i \in [b+2,k]$ for different values of $b$.
\item \textbf{Experiment 3}: $p_{1i}=1/k$, $p_{2i}=0$ for $i \in [1,b]$ and $p_{2i}=1/(k-b)$ for $i>b$ for different values of $b$.
\end{itemize}
For each experiment, we consider two configurations of sample sizes and dimensions:
 $(n_1, n_2, k)=(500,500, 10^3)$ and $(2000,2000, 10^4)$.

\begin{table}[ht]
\centering
\begin{tabular}{rrrr|rrrr}
\hline
 \multicolumn{4}{c}{$n_1=n_2 =500, k=10^3$} & \multicolumn{4}{c}{$n_1=n_2 =2000, k=10^4$} \\ \hline
   $m$ & $T$ & $BS$ & $Zel$          &  $m$   & $T$ & $BS$ & $Zel$ \\
  \hline
    1($H_0$) &0.042& 0.067& 0.041 &   1 ($H_0$)& 0.056& 0.087& 0.056\\
    2 &0.068& 0.105& 0.058     &   10    & 0.133& 0.183& 0.070\\
   10 &0.161& 0.202& 0.075     &   $10^2$& 0.282& 0.344& 0.102\\
  100 &0.292& 0.364& 0.135     &   $10^3$& 0.327& 0.389& 0.103\\
 1000 &0.363& 0.444& 0.153     &   $10^4$& 0.347& 0.410& 0.114\\
 \hline
\end{tabular}
\label{tab:Exp1}
\caption{Experiment 1. $m=1$ implies $H_0$.  }
\end{table}

\begin{table}[ht]
\centering
\begin{tabular}{rrrr|rrrr}
\hline
 \multicolumn{4}{c}{$n_1=n_2 =500, k=10^3$} & \multicolumn{4}{c}{$n_1=n_2 =2,000, k=10^4$} \\
  \hline
 $b$ & $T$ & $BS$ & $Zel$ & $b$ & $T$ & $BS$ & $Zel$ \\
  \hline                     \hline
    0& 0.033& 0.076& 0.041&    0& 0.061& 0.093& 0.061\\
   10& 0.183& 0.248& 0.103&   20& 0.092& 0.141& 0.064\\
   20& 0.588& 0.652& 0.146&   50& 0.493& 0.565& 0.108\\
   25& 0.779& 0.829& 0.193&   70& 0.811& 0.851& 0.153\\
   30& 0.915& 0.937& 0.255&  100& 0.978& 0.985& 0.210\\
   \hline
\end{tabular}
\label{tab:Exp2}
\caption{Experiment 2.  $b=0$ implies  $H_0$.  }
\end{table}

\begin{table}[ht]
\centering
\begin{tabular}{rrrr|rrrr}
\hline
 \multicolumn{4}{c}{$n_1=n_2 =500, k=10^3$} & \multicolumn{4}{c}{$n_1=n_2 =2,000, k=10^4$} \\
  \hline
 $b$ & $T$ & $BS$ & $Zel$  &   $b$ & $T$ & $BS$ & $Zel$ \\
  \hline
    0& 0.033& 0.076& 0.041 &      0& 0.061& 0.093& 0.061\\
  100& 0.142& 0.212& 0.152 &   1000& 0.171& 0.259& 0.186\\
  200& 0.363& 0.469& 0.429 &   2000& 0.483& 0.587& 0.531\\
  300& 0.700& 0.791& 0.784 &   3000& 0.890& 0.935& 0.912\\
  400& 0.962& 0.979& 0.981 &   4000& 0.998& 0.999& 0.998\\
   \hline
\end{tabular}
\label{tab:Exp3}
\caption{Experiment 3.  $b=0$ implies $H_0$.  }
\end{table}

Note that for the null hypothesis, we use $p_{1i}$s described in Experiments 1-3. Additionally,
we use $p_{2i}$s in Experiments 1-3 for the alternative.

Experiment 1 shows that
the probabilities, $p_{1i}$s, are decreasing in $i$ which is the case that
some cells have large counts and others have sparse counts.
For the situation of $H_1$,
only two entries (1st and $m$th in ${\bf P}_2$) are changed to have different probability vector from ${\bf P}_1$.
As  $m$ increases, the inhomogeneity of two groups also increases, which leads to larger powers of tests.
On the other hand, in Experiments 2 and 3, $p_{1i}$s all have equal probability $1/k$. 
For the $H_1$,  Experiments 2 and 3 use different configurations of ${\bf P}_2$. For example, 
 $p_{2,b+1}$ in Experiment 2 has very spiky values as $b$ increases while 
 Experiment 3   $p_{2i}$s have all the same values for $i>b$.   

Tables \ref{tab:Exp1} and \ref{tab:Exp2} provide the results of Experiment 1 and 2 showing  
that  $T$ have significant advantage over Zelterman's test in power while
the BS test tends to have larger sizes than the nominal level $.05$ as also shown in Tables \ref{tab:sizes1} and \ref{tab:sizes2}.
For Experiment 3,  Table \ref{tab:Exp3} shows that
 Zelterman's test seems to have slightly higher powers than the proposed test.
 The BS test has the highest powers among three tests, however
 the BS test has inflated sizes which lead to higher powers.

We additionally consider the following simulations for powers.
Experiment 4 and 5 use the cases that
sample sizes ($n_c$ for $c=1,2$) are four times the dimension ($4 \times k$) and
$k$ increases from $10^3$ to $10^5$. Note that sample sizes also increase at the linear rate of $k$.
\begin{itemize}
   \item \textbf{Experiment 4:} $p_{1i}=1/k$, $p_{2i}=0$ for $i \in [1,b]$ and $p_{2i}=1/(k-b)$ for $i>b$.
    Here we used $b=50$ and $n_c = 4k$ for $c=1,2$.
\item \textbf{Experiment 5 :}  $p_{1i}=\frac{1/i^\gamma}{\sum_{i=1}^k 1/i^\gamma}$, where $\gamma=0.45$.
 $n_c = 4k$ for $c=1,2$.
The probability vector for the $2^{nd}$ group was generated by copying the probability vector of the $1^{st}$ group and then switching the 1st and 5th entries of that vector.  $n_c = 4k$ for $c=1,2$.  \end{itemize}

Table  6 shows the powers of three tests for Experiment 4 and 5.
In Experiment 4, all three tests decrease 
as $k$ increases.
We can see that the Zelterman's test has the highest powers in Experiment 4. 
The BS test has the slightly higher powers than the proposed test, however this is due to the tendency that 
the BS test has inflate sizes. 
On the other hand, in Experiment 5, the proposed test and the BS test 
tend to have increasing powers as $k$ increases while the Zelterman's test
has decreasing pattern of powers. The BS test still has slightly more powers than the proposed test, but this is also due to inflated sizes of
the BS test.

\begin{table}[ht]
\centering
\begin{tabular}{rrrr|rrrr}
\hline
 \multicolumn{4}{c}{\textbf{Experiment 4}  } & \multicolumn{4}{c}{ \textbf{Experiment 5} } \\
  \hline
 $k$ & $T$ & $BS$ & $Zel$  &   $k$ & $T$ & $BS$ & $Zel$ \\
  \hline
       100& 1.000& 1.000& 1.000&       100& 0.377& 0.464& 0.145 \\
      1000& 0.736& 0.810& 0.983&      1000& 0.761& 0.812& 0.120 \\
      2000& 0.481& 0.571& 0.834&      2000& 0.876& 0.908& 0.112 \\
      3000& 0.364& 0.454& 0.681&      3000& 0.938& 0.956& 0.112 \\
     10000& 0.176& 0.240& 0.317&     10000& 0.996& 0.997& 0.108 \\
   \hline
\end{tabular}
\label{tab:Power1}
\caption{Powers from Experiment 4 and 5.  }
\end{table}

Lastly, we consider two more experiments, Experiment 6 and 7.
The dimension $k$ is more than the sample sizes such as $k=4 n_c$ and $n_1=n_2$.

\begin{itemize}
   \item \textbf{Experiment 6:} $p_{1i}=1/k$, $p_{2i}=0$ for $i \in [1,b]$ and $p_{2i}=1/(k-b)$ for $i>b$. Here we used $b=500$ and $k=4n_c$ for $c=1,2$.
\item \textbf{Experiment 7:}  $p_{1i}=\frac{1/i^\gamma}{\sum_{i=1}^k 1/i^\gamma}$, where $\gamma=0.45$.
 $k=4n_c$ for $c=1,2$.
The probability vector for the $2^{nd}$ group was generated by copying the probability vector of the $1^{st}$ group and then switching the 1st and 500th entries of that vector.
\end{itemize}

Table 7
shows the results of Experiment 6 and 7. We see similar results to Experiment 4 and 5.
In particular, Experiments 5 and 7, the Zelterman's test 
has drawback in obtaining powers while the proposed test and the BS test
have increasing power as $k$ increases.

\begin{table}[ht]
\centering
\begin{tabular}{rrrr|rrrr}
\hline
 \multicolumn{4}{c}{\textbf{Experiment 6}  } & \multicolumn{4}{c}{ \textbf{Experiment 7} } \\
  \hline
 $k$ & $T$ & $BS$ & $Zel$  &   $k$ & $T$ & $BS$ & $Zel$ \\
  \hline
      1000& 0.787& 0.868& 0.827&      1000& 0.175& 0.247& 0.111 \\
      2000& 0.346& 0.451& 0.371&      2000& 0.210& 0.279& 0.110 \\
      3000& 0.240& 0.328& 0.260&      3000& 0.249& 0.326& 0.110 \\
     10000& 0.109& 0.164& 0.116&     10000& 0.389& 0.472& 0.109 \\
   \hline
\end{tabular}
\label{tab:Power2}
\caption{Powers from Experiment 6 and 7.  }
\end{table}

In Experiment 5 and 7, the increasing pattern of powers of the proposed test can be explained through our result in Corollary \ref{cor:power}.
For given probabilities in Experiment 5 and 7,  $n$ and $k$ have linear relationships and  ${\bf P}_2$ is obtained by switching 
 two components in ${\bf P}_1$, so  we obtain the following result;
for given $m$ such that  $p_{21}=p_{1m}$, $p_{2m}=p_{1m}$ and $p_{1i}=p_{2i}$ for $i \neq 1, m$, then we have  
\begin{eqnarray*}
||{\bf P}_1 - {\bf P}_2||_2^2  &=& (p_{11} -p_{21})^2 + (p_{1m}-p_{2m})^2 \asymp  k^{2-2\gamma}\\
\sigma_k^2 &\asymp&   n^{-2} ( \frac{k^{1-2\gamma}}{k^{2-2\gamma}} +  n^{-1}) \asymp k^{-3}
\end{eqnarray*}
which leads to
\begin{eqnarray*}
\frac{||{\bf P}_1 - {\bf P}_2||_2^2}{\sigma_k}  \asymp  \frac{k^{-2+2\gamma}}{k^{-3/2}} = k^{-\frac{1}{2} + 2\gamma}.
\end{eqnarray*}
For $1/4 < \gamma< 1/2$,  we have $ \frac{||{\bf P}_1 - {\bf P}_2||_2^2 }{\sigma_k} \asymp  k^{-\frac{1}{2} +2\gamma} \rightarrow \infty$
which results in the convergence of power of $T$ and $T'$  to 1.  Since $\gamma=0.45>1/4$,
the powers of $T$ and $T'$ are increasing to 1 as $k$ increases.
If $ \gamma = 1/4$, then we have $  0<  \lim_k P(T >  z_{1-\alpha}) \leq  \limsup_k P(T>z_{1-\alpha})  <1$;
if $ 0\leq  \gamma < \frac{1}{4}$, we have  $P(T > z_{1-\alpha}) \searrow \alpha$, decreasing to the nominal Type I error $\alpha$
from Corollary \ref{cor:power}.
On the other hand, there is no study on the asymptotic power function in Zelterman (1987), so it is not easy to
investigate the behavior of power of the Zetlerman's test analytically.
Our simulation studies in Experiment 5 and 7 show that
the Zetlerman's test has decreasing pattern of powers as $k$ increases while the proposed test and the BS test
have increasing patterns of powers.

To summarize, the proposed test and Zelterman's test 
control a given level of size while 
the BS test tends to have inflated sizes which is the critical drawback of the BS test.  
The BS test has the highest powers all situations, however such high powers are not reliable due to inflated sizes.   
The Zelterman's test have slightly more advantage over the proposed test in powers  
in some cases (Experiment 3); otherwise
our proposed test has significantly more powers than the Zelterman's test from our simulation studies.
Overall, the proposed test is reliable in controlling the nominal level of size and
obtaining reasonable powers while Zelterman's test and the BS test
has drawback in either controlling the nominal level of size or obtaining powers.


\section{Real Example: 20 Newsgroups}
\label{newsGroups}

Next we'll illustrate the use of the proposed neighborhood test using our statistic $T$ and the popular 20 newsgroups dataset. This dataset, originally assembled by Ken Lang, consists of 20,000 documents each of which comes from one of 20 different newsgroups.
 We used the training set available at \verb+http://qwone.com/~jason/20Newsgroups/+.

We compared the group rec.sports.baseball with sci.med to test the null hypothesis that the 2 groups of documents come from the same newsgroup. The $i^{th}$ entry of the data vector contains the count of the $i^{th}$ dictionary word seen in the set of documents, where the dictionary is composed of all unique words seen in both sets of documents. We compose such a vector for each of the two groups.
{For testing  $H_0 : {\bf P}_1 = {\bf P}_2$,  we observe that $p$-values from all tests described in this paper are almost 0, therefore
two groups are obviously different.
In such a case, we consider neighborhood test based on $T$ as discussed in section \ref{NeighborhoodTest} since
$T$ has the closed form of power function in Corollary \ref{Theorem1}.}
In the provided data set, each group consisted of 594 documents. For each of 100 replications we sampled documents to calculate the power and size. To obtain power, we sampled 50 documents from each group  (and subsequently 100 and 200 documents as additional experiments).
For size of test, we sampled two groups of 50 documents from the same group (and subsequently 100 and 200 documents as additional experiments). The dimension, 16,214, was defined by the the set of unique words found in the two groups being compared. The results are shown in figure \ref{neighborhood200} where we show $\delta$ vs $p_{\delta}$ for both power and size. The three plot show three different sample sizes (50, 100, and 200 sampled documents per group). Notice that the null and alternative hypotheses become more separable as the number of documents increases.

%
%
%
%

\begin{figure}
  \begin{center}
    	\includegraphics[width=0.4\textwidth]{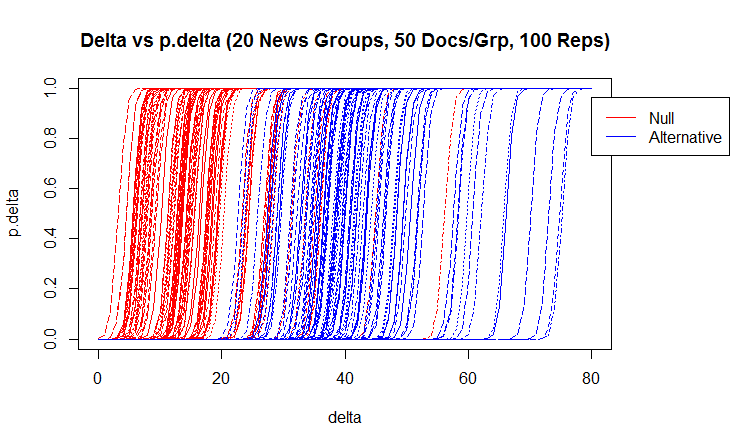}
%
    	\includegraphics[width=0.4\textwidth]{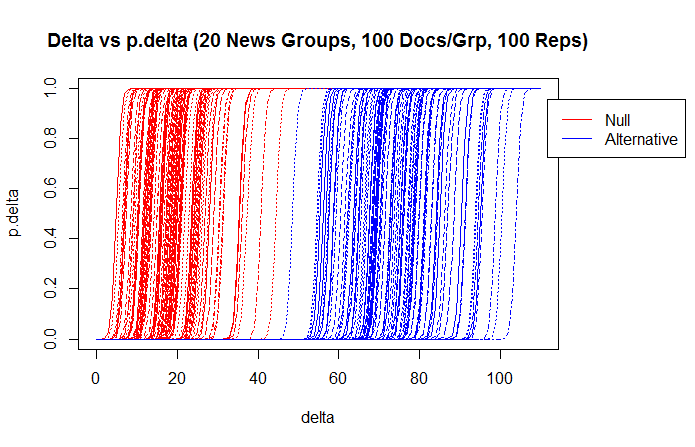}
%
    	\includegraphics[width=0.4\textwidth]{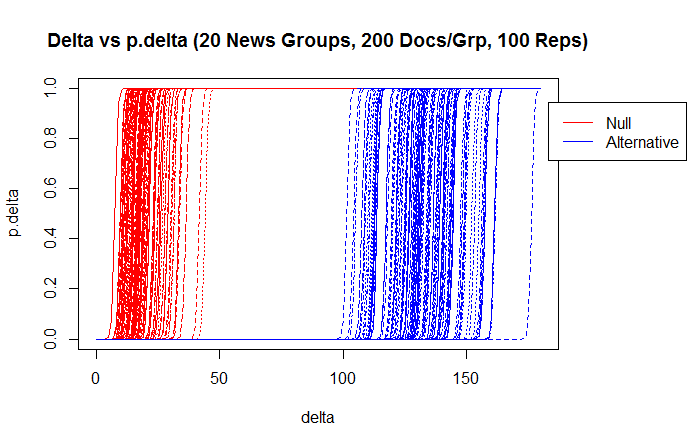}
\caption{$P$-value curve for various values of $\delta$.
50, 100 and   200 documents chosen per group.}
	\label{neighborhood200}
  \end{center}
\end{figure}

\section{Concluding Remaks}
\label{MultinomialConclusion}

In this paper we developed new statistics for testing the homogeneity of two probability vectors from two multinomial distributions and showed the asymptotic normality of the proposed tests under some regularity conditions. Through simulations we showed that our proposed test statistic  performs very well (i.e. have high power while controlling size) especially for situations where the data is sparse. In some cases the power of our new statistic was 3-4 times that of some existing test.
In Experiment 5 and 7 of the simulation studies we even saw that the power of our proposed test increased as dimension increased, while the power of the other method remained low. Additionally,
using the power function of our proposed test,
we discussed the use of a neighborhood test with our statistic as a means to make the test less sensitive to insignificant differences between the two groups.
We applied this neighborhood test to the popular 20 newsgroups data set to show that our test is effective in testing the null hypothesis that the groups of documents are from the same newsgroup.

\appendix 
\section*{Appendix}
\section{Proof of Lemma \ref{lemma:ratioconsistency}}
\label{ProofLemmaRatioConsistents}
We show the ratio consistency of $\hat \sigma_k^2$.
To show the ratio consistency of $\hat \sigma_k^2$, by using   $n_1 \asymp n_2$ and $\sigma_k^2 \asymp n^{-2} \vari$, it is sufficient to show
\begin{eqnarray}
\frac{\hat \sigma_k^2 -\sigma_k^2}{\sigma_k^2}  
&\asymp &  \frac{ \sumk (\hat p^2_{1i} -\frac{\hat p_{1i}}{n_1} -p_{1i}^2)}{\vari} \nonumber
+  \frac{\sumk \hat p_{1i}\hat p_{2i}  -\sumk p_{1i}p_{2i} }{\vari} + \frac{ \sumk (\hat p^2_{2i} -\frac{\hat p_{2i}}{n_2} -p_{2i}^2)}{\vari} 
\nonumber \\
&&\cp 0.
\label{eqn:convergence}
\end{eqnarray}
We first show the ratio consistency of $\sumk \left(\hat p_{1i}^2 -\frac{\hat p_{1i}}{n_1}\right)$ for $\sumk p_{1i}^2$. 
The case of the 2nd group ($\sumk \left(\hat p_{1i}^2 -\frac{\hat p_{1i}}{n_1}\right)$ for $\sumk p_{1i}^2$)  can be proved similarly. 
Since  $E(\hat p_{1i} ^2) = p_{1i}^2 + \frac{p_{1i}(1-p_{1i})}{n_1} = (1-\frac{1}{n_1})p_{1i}^2 + \frac{p_{1i}}{n_1}$
where $\hat p_{1i} = \frac{N_{1i}}{n_1}$, we have
$E \left( \frac{n_1}{n_1-1}(\hat p_{1i}^2 - \frac{\hat p_{1i}}{n_1}) \right) = p_{1i}^2$.
Thus we consider the following unbiased estimator of $\sum_{i=1}^k p_{1i}^2$:
$\frac{n_1}{n_1-1}\sum_{i=1}^k(\hat p_{1i}^2 - \frac{\hat p_{1i}}{n_1})$.
To show $\frac{\frac{n_1}{n_1-1}\sum_{i=1}^k(\hat p_{1i}^2 - \frac{\hat p_{1i}}{n_1}) - \sumk p^2_{1i}  }{\vari } \cp 0$,
we will show that the following quantity converges to 0 as follows:
\begin{eqnarray}
&&\frac{E \left[ \left( (\frac{n_1}{n_1-1})\sum_{i=1}^k (\hat p_{1i}^2 - \frac{\hat p_{1i}}{n_1}) -\sum_{i=1}^k p_{1i}^2 \right)^2 \right] }{\var{4}} = 
\frac{Var\left( (\frac{n_1}{n_1-1})\sum_{i=1}^k (\hat p_{1i}^2 - \frac{\hat p_{1i}}{n_1}) \right)}{\var{4}}  \nonumber \\
&\leq& \frac{8}{\var{4}} \left(Var(\sum_{i=1}^k \hat p_{1i}^2) + Var(\sum_{i=1}^k \frac{\hat p_{1i}}{n})\right) 
= \frac{8}{\var{4}} ((I) + (II))   \label{eqn:inequality}
\end{eqnarray}
where the last inequality in (\ref{eqn:inequality}) is from $Var(X+Y) \leq 2(Var(X)+Var(Y))$ and $n_1/(n_1-1) \leq 2$.
We decompose $(I)$ into two parts: 
\begin{eqnarray*}
(I)=Var(\sum_{i=1}^k \hat p_{1i}^2) &=& \sum _{i=1}^k Var(\hat p_{1i}^2) + \sum_{i \neq j} Cov(\hat p_{1i}^2, \hat p_{1j}^2) = (A) + (B).
\end{eqnarray*}
Using 
the results in Lemma.S2 in Supplementary material,  for some constants $C_1$ and $C_2$, we have 
\begin{eqnarray*}
(A)&=& \sum_{i=1}^k \left( E(\hat p_{1i}^4) - (E(\hat p_{1i}^2))^2 \right) \leq  C_1 \sum_{i=1}^k  \left( \frac{p_{1i}^4}{n_1} + \frac{p_{1i}^3}{n_1} + \frac{p_{1i}^2}{n_1^2} +\frac{p_{1i}}{n_1^3} \right)  \\
|(B)|&=& \left|\sum_{i \neq j} \left( E(\hat p_{1i}^2 \hat p_{1j}^2) -E(\hat p_{1i}^2) E(\hat p_{1j}^2) \right) \right|  \leq  C_2 \sum_{i\neq j}  \left( \frac{p_{1i}^2 p_{1j}^2}{n_1} +
 \frac{p_{1i}^2 p_{1j}}{n_1^2} + \frac{p_{1i}p_{1j}^2}{n_1^2} +\frac{p_{1i}p_{1j}}{n_1^3} \right).
\end{eqnarray*}
For all the terms in the above, we can show  $\frac{(A)}{\var{4}} \rightarrow 0$ and $\frac{(B)}{\var{4}} \rightarrow 0$ as follows:
first, note that   $ \frac{\max_i p_{ci}^2}{\vari}  \rightarrow 0$ since $ \frac{\max_i p_{ci}^2}{\vari}   \leq \frac{\max_i p_{ci}}{||{\bf P}_c ||_2^2} \rightarrow 0$ from the condition 2 in Theorem \ref{Theorem1}. 
For $(A)$, using $\max_i p_{1i}^2 \leq \max_i p_{1i} \rightarrow 0$ in the result 2 in Lemma.S2 in the Supplementary material, 
we have 
\begin{eqnarray*}
\frac{\sum_{i=1}^k p_{1i}^4}{n \var{4}} &\leq&  \frac{\max_i p_{1i}^2 \sum_{i=1}^k p_{1i}^2}{ n_1 \var{4} } = \frac{\max_i p_{1i}^2}{n_1 \vari} \rightarrow 0,\\
\frac{\sum_{i=1}^k p_{1i}^3}{n \var{4}} &\leq&  \frac{\max_i p_{1i} \sum_{i=1}^k p_{1i}^2}{ n_1 \var{4} } = \frac{\max_i p_{1i}}{n_1} \rightarrow 0,\\
\frac{\sum_{i=1}^k p_{1i}^2}{n_1^2 \var{4}} &\leq&  \frac{1}{ n_1  (n_1 \vari) } \rightarrow 0,~~\frac{\sum_{i=1}^k p_{1i}}{n_1^3 \var{4} } \leq  \frac{1}{ n_1  (n_1\vari)^2 } \rightarrow 0
\end{eqnarray*}
where the condition 3 ($ n\vari \geq \epsilon >0 $) and $n_1 \asymp n_2 $ are used  in the last steps  as $n_1 \rightarrow \infty$.     
For $(B)$, using $\sum_{i\neq j} p_{1i}^2 p_{1j}^2 \leq \vari$
and $\sum_{i\neq j} p_{1i} p_{1j} \leq \sum_{i=1}^k p_{1i} =1$, we have from the conditions 1-3 in Theorem \ref{Theorem1}
\begin{eqnarray*}
\frac{\sum_{i \neq j} p_{1i}^2 p_{1j}^2 }{n_1 \var{4}} &\leq& \frac{1}{n_1} \rightarrow 0,~~~~~~\frac{\sum_{i \neq j} p_{1i}^2 p_{1j} }{n_1^2 \var{4}} \rightarrow 0, \\
\frac{\sum_{i \neq j} p_{1i} p_{1j}^2 }{n_1^2 \var{4}} &\leq&  \frac{\vari }{n_1^2 \var{4} } \le
\frac{1}{n_1 \vari }   \rightarrow 0, \\
\frac{\sum_{i \neq j} p_{1i} p_{1j} }{n_1^3 \var{4}} &\leq& \frac{1}{n_1 (n_1\vari)^2} \rightarrow 0. \label{here}
\end{eqnarray*}
Similarly, for (II),  we have
\begin{eqnarray*}
(II) =
Var(\sum_{i=1}^k \frac{\hat p_{1i}}{n_1}) &=&  \sum_{i=1}^k Var(\frac{\hat p_{1i}}{n_1}) + \frac{1}{n_1^2} \sum_{i \neq j} Cov(\hat p_{1i}, \hat p_{1j}) \\
    &=& \sum_{i=1}^k \frac{p_{1i}(1-p_{1i})}{n_1^3} - \sum_{i\neq j} \frac{p_{1i}p_{1j}}{n_1^3} \leq \sumk \frac{p_{1i}}{n_1^3} =\frac{1}{n_1^3}.
\end{eqnarray*}
Therefore, we have $\frac{(II)}{n_1^3\var{4}}  \leq \frac{1}{n_1^3\var{4}}  \rightarrow 0$ which leads
\begin{eqnarray}
 \frac{ \sumk (\hat p^2_{1i} -\frac{\hat p_{1i}}{n_1})  -\sumk p_{1i}^2}{\vari} \cp 0.
 \label{eqn:square}
\end{eqnarray}
The ratio consistent estimator of $\sumk \left(\hat p_{2i}^2 - \frac{\hat p_{2i}}{n_2}\right)$ can be also proved in the same way.    

For  $\sumk \hat p_{1i}\hat p_{2i}$, we show
\begin{eqnarray*}
\frac{E((\sumk \hat p_{1i}\hat p_{2i})^2)}{\var{4}}
&=& \frac{\sum_{i=1}^k Var(N_{1i}N_{2i}) + \sum_{i\neq j}\rcov(N_{1i}N_{2i},N_{1j}N_{2j})}{\var{4}} \\
&\asymp&  \frac{\sum_{i=1}^k p_{1i}^2 p_{2i}}{n \var{4} }+  \frac{\sum_{i=1}^k p_{1i} p_{2i}^2}{n\var{4}} + \frac{1}{n^2 \vari}  \\
&&- \frac{\sum_{i\neq j}p_{1i}p_{2i}p_{1j}p_{2j}}{n\var{4}} \\
&\asymp &  \frac{\max_i p_{1i} }{n\vari} +\frac{\max_i p_{2i} }{n\vari} + \frac{\sum_{i=1}^k p_{1i} p_{2i}^2}{n\var{4}}\\
&& + \frac{1}{n^2 \vari}
- \frac{ ({\bf P}_1\cdot {\bf P}_2)^2 }{n\var{4}} \rightarrow 0
\end{eqnarray*}
where the last term converges to 0 since $\frac{ ({\bf P}_1\cdot {\bf P}_2)^2 }{n\var{4}} \leq \frac{\vari}{2n\var{4}}  = \frac{1}{2n\vari} \rightarrow 0$ from the condition 3 in Theorem \ref{Theorem1}. Therefore
\begin{eqnarray}
\frac{\sumk \hat p_{1i}\hat p_{2i} - \sumk p_{1i}p_{2i} }{ \vari} \cp 0
\label{eqn:product}
\end{eqnarray}

Combining \eqref{eqn:square} and \eqref{eqn:product}, we have  \eqref{eqn:convergence} which leads to
the ratio consistency of $\hat \sigma_k^2$.

\newpage
{\Large \bf Supplementary Material}
\section{Supplementary Lemmas}
{\bf Lemma.S1} \\
If conditions 1-4 in Theorem 1 are satisfied, we have the following results.
\begin{enumerate}
\item
$\max_{1 \le i \le k} p_{ci} \rightarrow 0$ for $c=1,2$.
\item
$\frac{\maxp }{\vari} \rightarrow 0.$
\item   $ n ||\boldsymbol{\xi} * ({{\bf P}_1} + {{\bf P}_2}) ||_2^2  =O(\var{4}) $.
\item  $ |{\boldsymbol \xi}| \cdot({\bf P}_1 + {\bf P}_2) =O(\vari)$.
\item  $n(|\boldsymbol \xi|\cdot {\bf P}_1) (|\boldsymbol \xi|\cdot {\bf P}_2)= O(\var{4})$.
\item $n ( |\boldsymbol \xi|\cdot({\bf P}_1  + {\bf P}_2))^2 =O(\var{4})$.
\item $\sqrt{n}|| \boldsymbol{\xi}*(\sqrt{{\bf P}_1} + \sqrt{{\bf P}_1}) ||_2^2 = \frac{1}{\sqrt{n}}O(\var{2})$.
\end{enumerate}
\begin{proof}
\begin{enumerate}
\item   Result 1 can be shown by contradiction.
Assume  $\frac{\max_i p_{ci}^2}{||{\bf P}_c||_2^2} \rightarrow 0$  holds but $\max_{1 \le i \le k} p_{ci} \nrightarrow 0$ as $k \rightarrow \infty$.
Then, there exist a subsequence $\{k_1', k_2',\ldots, \} \subset \{1,2,\ldots, \}$ such that
$\max_{1 \le i \le k_n'} p_{ci} > \epsilon$ for some $\epsilon>0$.
Since $ ||{\bf P}_c||_2^2= \sum_{i=1}^k p_{ci}^2 \leq 1$ from $p_{ci}^2 \leq p_{ci}$,
 we have
$  \frac{\max_{1 \le i \le k_n'} p_{ci}^2}{ ||{\bf P}_c||_2^2 }
 \geq \epsilon^2 $
 for the sequence $k_n' \rightarrow \infty$.
This is a contradiction to  $\frac{\max_i p_{ci}^2}{||{\bf P}_c||_2^2} \rightarrow 0$.
Therefore, we have $\max_{1\leq i\leq k} p_{ci} \rightarrow 0$.

\item  From $ ||{\bf P}_1 ||_2^2 + ||{\bf P}_2 ||_2^2 \geq 2 ||{\bf P}_1 ||_2||{\bf P}_2 ||_2$, we have
$\frac{\maxp}{\vari}
\leq  \frac{1}{2} \frac{\max_i p_{1i} \max_{i} p_{2i}}{||{\bf P}_1 ||_2||{\bf P}_2 ||_2}
\leq \frac{C}{2} \sqrt{\frac{\max_{1 \le i \le k} p_{1i}^2}{||{\bf P}_1||_2^2}}\sqrt{\frac{\max_{1 \le i \le k} p_{2i}^2}{||{\bf P}_2||_2^2} } \rightarrow 0$
from 1 in this Lemma.
\item
$n || \boldsymbol{\xi} * ({\bf P}_1 + {\bf P}_2) ||_2^2   \leq   n ||\boldsymbol{\xi} ||_2^2  \vari
= O(\var{4})$
where the last equality is from the condition 4 in Theorem 1.
\item
$|\boldsymbol{\xi}|\cdot({\bf P}_1 + {\bf P}_2) = \sumk |p_{1i}-p_{2i}|(p_{1i}+p_{2i}) \leq
\sumk (p_{1i}+p_{2i})^2 = \vari$.
\item By Cauchy-Schwartz inequality, we have
 $n(|\boldsymbol \xi|\cdot {\bf P}_1) (|\boldsymbol \xi|\cdot {\bf P}_2)
\leq  n || {\boldsymbol \xi} ||_2^2  ||{\bf P}_1 ||_2||{\bf P}_2 ||_2 =  O(\var{4})$.
\item Using Cauchy-Schwartz inequality,
$ |\boldsymbol \xi|\cdot ({\bf P}_1 + {\bf P}_2) \leq  || {\boldsymbol \xi} ||_2^2  \vari
= O(\var{4})=O(\vari)$.
\item
$\sqrt{n} ||\boldsymbol{\xi}*(\sqrt{{\bf P}_1}+\sqrt{{\bf P}_2})||_2^2 \leq \sqrt{n}
=||\boldsymbol{\xi} ||_2^2 (\sumk(p_{1i}+p_{2i}))^2 = \frac{1}{\sqrt{n}}O(\vari)$
from the condition 4 in Theorem 1.
\end{enumerate}
\end{proof}

\label{HigherOrderMoments}
 The following higher order moments of the multinomial distribution are given by  Newcomer et al. (2008).
{\bf Lemma.S2}
Let $(N_1,N_2,...,N_{k})$ be a $k$-dimensional multinomial random variable with parameters ${\bm p}=(p_1,p_2,...,p_{k})$ and sample size $n$. Also let
$x^{(a)}=x(x-1)...(x-a+1)$. Then we have the following moments:

\begin{enumerate}
\item $E(N_i) = np_i$
\item $E(N_iN_j) = n^{(2)}p_ip_j , \forall i \ne j$.
\item $E(N_i^2) = n^{(2)}p_i^2 + np_i$.
\item $E(N_i^2N_j) =  n^{(3)}p_i^2p_j + n^{(2)}p_ip_j, \forall i \ne j$.
\item $E(N_i^3) = n^{(3)}p_i^3 + 3n^{(2)}p_i^2 + np_i = O(n^3 p_i^3 + np_i)$.
\item $E(N_i^2N_j^2) = n^{(4)} p_i^2p_j^2 + 3n^{(3)}(p_i^2p_j + p_ip_j^2) + n^{(2)}p_ip_j =
O(n^4 p_i^2 p_j^2 + n^2 p_i p_j), \forall i \ne j$.
\item $E(N_i^4) = n^{(4)}p_i^4 + 6n^{(3)}p_i^3 + 7n^{(2)}p_i^2 + np_i = O(n^4 p_i^4 + np_i)$.
\end{enumerate}

\section{Proof of Lemma 2}
\label{lemmaProof}



\begin{enumerate}
\item When $X_{1i}$ and $X_{2i}$ are independent Poisson with
$Poisson(\lambda_{1i})$ and $Poisson(\lambda_2i)$ for $\lambda_{1i}=n_1p_{1i}$ and $\lambda_{2i}=n_2p_{2i}$, we have
\begin{eqnarray*}
E(f_i(X_{1i}, X_{2i})) &=&  E ( \frac{X_{1i}}{n_1} -\frac{X_{2i}}{n_2})^2 - E(\frac{X_{1i}}{n_1^2} + \frac{X_{2i}}{n_2^2})
- 2( p_{1i} -p_{2i}) E\left( \frac{X_{1i}}{n_1} -\frac{X_{2i}}{n_2} \right) + (p_{1i} -p_{2i} )^2 \\
&=& \frac{\lambda_{1i}}{n_1^2} + \frac{\lambda_{2i}}{n_2} + \left(\frac{\lambda_{1i}}{n_1} - \frac{\lambda_{2i}}{n_2} \right)
-  \left( \frac{\lambda_{1i}}{n_1^2} + \frac{\lambda_{2i}}{n_2}\right)
  -    2( \frac{\lambda_{1i}}{n_1} - \frac{\lambda_{2i}}{n_2})^2 + ( \frac{\lambda_{1i}}{n_1} - \frac{\lambda_{2i}}{n_2})^2  =0
\end{eqnarray*}
\item
Using the independence of $X_{1i}$ and $X_{2i}$, we have
\allowdisplaybreaks
\begin{align}
\rcov(f_i (X_{1i},X_{2i}), X_{1i}) &= \rcov \left( \left( \frac{X_{1i}}{n_1} - \frac{X_{2i}}{n_2} \right)^2 - \frac{X_{1i}}{n_1^2}- \frac{X_{2i}}{n_2^2}
-2(p_{1i}-p_{2i})\left(\frac{X_{1i}}{n_1} - \frac{X_{2i}}{n_2} \right) , X_{1i}  \right) \nonumber \\
&= \rcov \left(  \frac{X_{1i}^2}{n_1^2} - \frac{2X_{1i}X_{2i}}{n_1n_2} -\frac{X_{1i}}{n_1^2} -2(p_{1i}-p_{2i})\frac{X_{1i}}{n_1}    , X_{1i}  \right) \nonumber \\
&= 2p_{1i}(p_{1i}-p_{2i}) + \frac{p_{1i}}{n_1} - \frac{p_{1i}}{n_1}- 2(p_{1i}-p_{2i})p_{1i}  \nonumber \\
&= 0. \label{cov_f_x}
\end{align}
To obtain  \eqref{cov_f_x}, we use
\begin{eqnarray*}
\rcov \left( \left( \frac{X_{1i}}{n_1} - \frac{X_{2i}}{n_2} \right)^2 , X_{1i} \right) &=& 2p_{1i}(p_{1i}-p_{2i}) + \frac{p_{1i}}{n_1}  \\
\rcov (X_{1i}^2,X_{1i}) &=&  2n_1^2 p_{1i}^2 + n_1 p_{1i} \\
\rcov (X_{1i} X_{2i}, X_{1i}) &=&  n_1 n_2 p_{1i} p_{2i}
\end{eqnarray*}
Similarly, we also obtain $\rcov(f_i(X_{1i}, X_{2i}), X_{2i}) = 0$.

\item Next we calculate $s_i^2= Var(f_i(X_{1i},X_{2i}))$, which is needed for the calculation of $\sigma^2_k=\sum_{i=1}^k s_i^2$.
Let $f^*_i (X_{1i},X_{2i}) = \left(\frac{X_{1i}}{n_1} - \frac{X_{2i}}{n_2} \right)^2 - \frac{X_{1i}}{n_1^2}
-\frac{X_{2i}}{n_2^2} $ and   $\xi_i = p_{1i}-p_{2i}$, then
\begin{eqnarray*}
s_i^2 &=& Var\left(  f^*(X_{1i}, X_{2i} )\right) + 4 \xi_i^2 Var\left(\frac{X_{1i}}{n_1} + \frac{X_{2i}}{n_2} \right)
-4 \xi_i Cov \left( f^*(X_{1i},X_{2i}),  \frac{X_{1i}}{n_1} -\frac{X_{2i}}{n_2}     \right) \\
&=& 2\left(\frac{p_{1i}}{n_1} + \frac{p_{2i}}{n_2} \right)^2
\end{eqnarray*}
where
\begin{eqnarray}
Var\left( \left( \frac{X_{1i}}{n_1} - \frac{X_{2i}}{n_2} \right)^2 - \frac{X_{1i}}{n_1^2} -\frac{X_{2i}}{n_2^2} \right)
&=& 2\sumk \left(\frac{p_{1i}}{n_1} + \frac{p_{2i}}{n_2} \right)^2 + 4 \sumk \xi_i^2 \left(\frac{p_{1i}}{n_1} +\frac{p_{2i}}{n_2} \right)\\
Var\left( \frac{X_{1i}}{n_1^2} + \frac{X_{2i}}{n_2^2}   \right)&=&  \frac{p_{1i}}{n_1}  + \frac{p_{2i}}{n_2} \\
-4 \xi_i Cov\left(f^*(X_{1i},X_{2i}),  \left(\frac{X_{1i}}{n_1} - \frac{X_{2i}}{n_2}  \right)  \right)
&=& -8 \xi_i^2 \left(\frac{p_{1i}}{n_1} + \frac{p_{2i}}{n_2} \right).
\end{eqnarray}

\end{enumerate}
Therefore, $\sigma_k^2 =\sumk s_i^2 =  2 \sumk \left(\frac{p_{1i}}{n_1} + \frac{p_{2i}}{n_2}  \right)^2$.
Additionally,  since $f(X_{1i},X_{2i}) = f^*(X_{1i}, X_{2i})$ under $H_0$,
we have $\sumk  Var(f^*(X_{1i},X_{2i})) =   2 \sumk \left(\frac{p_{1i}}{n_1} + \frac{p_{2i}}{n_2}  \right)^2$ under $H_0$.

\section{Proof of Lemma 3}
\label{Apdx:lemmaMorrisCondition}

Let us first find $f_i(L_{1i}+M_{1i},L_{2i}+M_{2i}) - f_i(L_{1i},L_{2i})$ using the Taylor Series. The general form of the Taylor expansion is:
\begin{align*}
f_i(x+h_1,y+h_2) &- f_i(x,y) \nonumber \\
=  f_{i,x}(x,y)h_1 &+ f_{i,y}(x,y)h_2 + \frac{1}{2!}f_{i,xx}(x,y)h_1^2 + \frac{1}{2!}f_{i,yy}(x,y)h_2^2 + \frac{2}{2!}f_{i,xy}(x,y)h_1h_2
\end{align*}
where $f_{i,x} = \frac{\partial }{\partial x}f$, $f_{i,xx}= \frac{\partial }{\partial^2x}f$ and others are similarly defined.
Note that there is no remainder term from Taylor expansion since $f_i$ is a quadratic function.

Using this formula and the definition of $f$ given in (19), we have:
\begin{eqnarray*}
&&f_{i}(L_{1i} + M_{1i},L_2+M_{2i}) - f_i(L_{1i},L_{2i}) \\
	&=&
\underbrace{2\left( \frac{L_{1i}}{n_1} - \frac{L_{2i}}{n_2} \right) \left( \frac{M_{1i}}{n_1} - \frac{M_{2i}}{n_2} \right) + \left( \frac{M_{1i}}{n_1} - \frac{M_{2i}}{n_2} \right)^2}_{Q_i}
\underbrace{	-  \left( \frac{1}{n_1^2} +  \frac{2(p_{1i}-p_{2i})}{n_1} \right)M_{1i} - \left(\frac{1}{n_2^2} +
 \frac{2(p_{2i}-p_{1i})}{n_2} \right)M_{2i}}_{R_i}   \label{TaylorResult}
 \end{eqnarray*}
where $A_i = \left( \frac{L_{1i}}{n_1} - \frac{L_{2i}}{n_2} \right)$, $B_i = \left( \frac{M_{1i}}{n_1} - \frac{M_{2i}}{n_2} \right)$, $Q_i=2A_iB_i + B_i^2$ and $R_i= -  \left( \frac{1}{n_1^2} +
 \frac{2(p_{1i}-p_{2i})}{n_1} \right)M_{1i} - \left(\frac{1}{n_2^2} +
 \frac{2(p_{2i}-p_{1i})}{n_2} \right)M_{2i}$.
We need to show
\begin{eqnarray*}
\frac{1}{\sigma_k^2}  E\left( \sumk (Q_i + R_i)\right)^2 &=&
\underbrace{\frac{1}{\sigma_k^2}\sumk E(Q_i^2)}_{W_1}  +
\underbrace{\frac{1}{\sigma_k^2} \sumk E(R_i^2)}_{W_2} +
\underbrace{\frac{1}{\sigma_k^2}\sum_{i\neq j}E (Q_i Q_j)}_{W_3}
+ \underbrace{\frac{2}{\sigma_k^2}\sum_{i\neq j}E (Q_i R_j)}_{W_4}
+ \underbrace{\frac{1}{\sigma_k^2}\sum_{i\neq j}E (R_i R_j)}_{W_5} \\
&\rightarrow& 0.
\end{eqnarray*}

\begin{enumerate}
\item
We show $   W_1= \frac{1}{\sigma_k^2} E \sumk Q_i^2 \rightarrow 0$.
\begin{eqnarray*}
	E(Q_i^2) &=& E \left[ (2A_iB_i + B_i^2)^2 \right] =4 \underbrace{ E(A_i^2)}_{(I)} \underbrace{E(B_i^2 )}_{(II)}
+ 4 \underbrace{E(A_i)}_{(III)} \underbrace{E(B_i^3)}_{(IV)} + \underbrace{E(B_i^4 )}_{(V)}. \nonumber
\end{eqnarray*}
We'll look at (I)-(V) separately below.
Since  $L_{1i}, L_{2i}, M_{1i}$ and $M_{2i}$ are independent,
 and $\max_{1\leq i\leq k}p_{ci} =o(1)$ for $c=1,2$, we have
\begin{eqnarray*}
	(I) &=& E \left[ \left( \frac{L_{1i}}{n_1} - \frac{L_{2i}}{n_2} \right)^2 \right]
                = Var \left(\frac{L_{1i}}{n_1} - \frac{L_{2i}}{n_2} \right) + \left[ E\left(\frac{L_{1i}}{n_1} - \frac{L_{2i}}{n_2} \right) \right]^2 \nonumber \\
&\leq& \left(\frac{p_{1i}}{n_1}+ \frac{p_{2i}}{n_2} \right) (1+o(1)) + 2 v^2 \left(\frac{p_{1i}^2}{n_1}+ \frac{p_{2i}^2}{n_2}  \right) \nonumber \\
&=& O\left(\frac{p_{1i}}{n_1}+ \frac{p_{2i}}{n_2} \right) = O\left( \frac{p_{1i} + p_{2i}}{n}\right)
\end{eqnarray*}

where $o(\cdot)$ and $O(\cdot)$ are uniform in $i$ and the last equality is from $ n_1/n \rightarrow C \in (0,1)$.
Similarly, we obtain
\begin{eqnarray*}
	(II) &=& E \left[ \left( \frac{M_{1i}}{n_1} - \frac{M_{2i}}{n_2} \right)^2 \right]
	= h O\left(  \frac{p_{1i} + p_{2i}}{n^{3/2}} \right) (1+o(1))\\
 (III) & = & E \left( \frac{L_{1i}}{n_1} - \frac{L_{2i}}{n_2}   \right)
\leq |p_{1i}-p_{2i}| = |\xi_i|.
\end{eqnarray*}
where $o(\cdot)$ is uniform in $i$.
Using Jensen's inequality, we have
\begin{eqnarray*} \label{eqnIV}
(IV) &=& E \left[ \left(   \frac{M_{1i}}{n_1}- \frac{M_{2i}}{n_2}  \right)^3 \right] \leq  2^2 \left(E \left( \frac{M_{1i}^3}{n_1^3}\right) + E \left( \frac{M_{2i}^3}{n_2^3} \right) \right)
= h O \left( \frac{p_{1i}^3 + p_{2i}^3}{n^{1.5}} + \frac{p_{1i} + p_{2i}}{n^{2.5}}  \right) \\
(V) &=&  E \left[ \left( \frac{M_{1i}}{n_1}- \frac{M_{2i}}{n_2}\right)^4 \right] \leq  2^3 \left(E \left( \frac{M_{1i}^4}{n_1^4}\right) + E \left( \frac{M_{2i}^4}{n_2^4} \right)  \right)
=   h O \left( \frac{p_{1i}^4+p_{2i}^4}{n^2}+\frac{p_{1i}+p_{2i}}{n^{3.5}}  \right). \label{eqnV}
\end{eqnarray*}

As the next step in showing that Equation~(23) holds, we need to sum over $k$ terms, divide by $\sigma_k^2
\asymp   n^{-2} \vari$, and show convergence to 0.

\begin{eqnarray*} \label{EQ2}
&&\frac{1}{\sigma_k^2}  \sum_{i=1}^k E(Q_i^2) =
h O \left(\frac{\sum_{i=1}^k (p_{1i}+p_{2i})^2}{n^{1/2} \vari} \right) + O\left(  \sumk \left(\frac{p_{1i}^{3.5} + p_{2i}^{3.5}}{\vari}
+ \frac{ (p_{1i}^{1.5} + p_{2i}^{1.5})\xi_i}{n \vari} \right) \right) \nonumber \\
&=& h O \left(\frac{ \psum  }{n^{1/2}  \vari} \right)
+ O\left( (\max p_{1i}^{1.5} +\max p_{2i}^{1.5}  )    \frac{\psum}{\vari}
+ \frac{(\max p_{1i}^{.5} + \max p_{2i}^{.5})|\boldsymbol{\xi}\cdot({\bf P}_1 +{\bf P}_2 )|}{n \vari}  \right) \\
&=& o(1)
\end{eqnarray*}
using $  |\boldsymbol{\xi} \cdot ({\bf P}_1 + {\bf P}_2) |  = O(\vari)$.
\item  We show $ W_2 = \frac{1}{\sigma_k^2} E (\sumk R_i^2) \rightarrow 0$.
Let $\xi_{i}= (p_{1i}-p_{2i})$ and define
$\kappa_{1i}=\frac{1}{n_1^2} a+ \frac{2(p_{1i}-p_{2i})}{n_1}=\frac{1}{n_1^2}+\frac{2\xi_i}{n_1}$ and
$\kappa_{2i}=\frac{1}{n_2^2}-\frac{2\xi_i}{n_2}$. Then $R_i= -\kappa_{1i}M_{1i} -\kappa_{2i}M_{2i}$ and using Jensen's inequality we have
\begin{eqnarray*}
E(R_i^2) &=& E \left[ \left( \kappa_{1i}M_{1i} + \kappa_{2i}M_{2i}\right)^2 \right]  \leq 2 \kappa_{1i}^2 E(M_{1i}^2) + 2\kappa_{2i}^2 E(M_{2i}^2)   \nonumber \\
&=& O\left( \frac{h(p_{1i}+p_{2i})}{n_1^{3.5}} + \frac{h\xi_i^2 (p_{1i}+p_{2i})}{n_1^{1.5}} + \frac{h^2 (p_{1i}^2+p_{2i}^2)}{n_1^3} + \frac{h^2 \xi_i^2 (p_{1i}^2+p_{2i}^2)}{n_1} \right) \nonumber \\
\end{eqnarray*}
Again, we need to sum over $k$ terms and divide by $\sigma_k^2$ and obtain
\begin{eqnarray*}
\label{ER2}
\frac{4}{\sigma_k^2}  \sum_{i=1}^k E(R_i^2) &=&  O\left( \frac{h\sum_{i=1}^k (p_{1i}+p_{2i})}{n^{1.5}\vari} \right) + O\left(\frac{ \sqrt{n} h\sum_{i=1}^k \xi_i^2 (p_{1i}+p_{2i})}{\vari} \right) \nonumber \\
&& + O\left(\frac{ h^2(\psum)}{n \vari}\right) +
 O\left(\frac{ n h^2\sum_{i=1}^k \xi_i^2 (p_{1i}^2+p_{2i}^2)}{\vari} \right) \nonumber \\
 &=& o(\frac{h}{\sqrt{n}}) + h  O\left( \frac{\sqrt{n}|| \boldsymbol{\xi}*(\sqrt{{\bf P}_1} +\sqrt{{\bf P}_2}) ||_2^2  }{ \vari} \right) + O( \frac{h^2}{n}) +
 h   O\left( \frac{n|| \boldsymbol{\xi}*({{\bf P}_1} +{{\bf P}_2}) ||^2_2  }{\vari}  \right) \nonumber \\
&=&  o(h) + \underbrace{O(h)}_{(*)} + o(h) + \underbrace{O(h)}_{(**)}  \label{secondlast}  \\
&=& o(1)~~~~\mbox{as $h \rightarrow 0$} \nonumber
\end{eqnarray*}
where the second term $(*)$ and  the fourth term $(**)$ are from  8 and 4 in Lemma.S, respectively.

\item We show $W_3= \frac{1}{\sigma_k^2}E( \sum_{i \neq j}Q_iQ_j) \rightarrow 0$.
We first have
\begin{align*}
\frac{1}{\sigma_k^2}\sum_{i\neq j} E(Q_iQ_j) &=  \sum_{i\neq j} E\left[ (2A_iB_i + B_i^2)(2A_jB_j + B_j^2) \right]  \nonumber \\
&=\underbrace{ \frac{4}{\sigma_k^2} \sum_{i\neq j} E(A_iA_j)E(B_iB_j)}_{H_1} +  \underbrace{\frac{4}{\sigma_k^2} \sum_{i \neq j}E (A_i)E(B_iB_j^2)}_{H_2}  +
\underbrace{\frac{4}{\sigma_k^2}\sum_{i\neq j} E(B_i^2B_j^2)}_{H_3} \nonumber \\
\end{align*}
We'll look at each term separately below.
\begin{align*}
E(A_iA_j)&= E \left[ \frac{L_{1i}L_{1j}}{n_1^2} - \frac{L_{1i}L_{2j}}{n_1n_2} - \frac{L_{2i}L_{1j}}{n_1n_2} + \frac{L_{2i}L_{2j}}{n_2^2}   \right]  \nonumber \\
	&= p_{1i}p_{1j} \frac{(n_1+u_1 n_1^{1/2})(n_1+u_1 n_1^{1/2}-1)}{n_1^2} - (p_{1i}p_{2j} + p_{1j}p_{2i})\frac{(n_1+u_1 n_1^{1/2})(n_2+u_2 n_2^{1/2})}{n_1n_2} \nonumber \\
	&+ p_{2i}p_{2j} \frac{(n_2+u_2 n_2^{1/2})(n_2+u_2 n_2^{1/2}-1)}{n_2^2}\nonumber \\
&=  O \left( |\xi_i||\xi_j| + u\left( \frac{ |\xi_i||\xi_j|}{n_1^{.5}}   \right)  + u^2 \left(\frac{ |\xi_i||\xi_j| }{n_1} \right) \right)
~~~~~\mbox{where $u=\max(u_1,u_2)$} \nonumber \\
&= O \left( |\xi_i||\xi_j| \left(  1 + \frac{u}{n_1^{.5}}  \right)^2 \right) = O(|\xi_i||\xi_j|)\\
E(B_i B_j) =&  E \left[ \frac{M_{1i}M_{1j}}{n_1^2} - \frac{M_{1i}M_{2j}}{n_1n_2} - \frac{M_{2i}M_{1j}}{n_1n_2} + \frac{M_{2i}M_{2j}}{n_2^2}   \right]  \nonumber \\
&= p_{1i}p_{1j} \frac{h_1 n_1^{1/2}(h_1n_1^{1/2}-1)}{n_1^2} - (p_{1i}p_{2j} +p_{1j}p_{2i} ) \frac{h_1h_2n_1^{1/2}n_2^{1/2}}{n_1n_2} +
p_{2i}p_{2j}\frac{h_2 n_2^{1/2}(h_2 n_2^{1/2}-1)}{n_2^2}\nonumber \\
&= h^2 O\left( \frac{|\xi_i||\xi_j|}{n}   \right)
\end{align*}
where $h=\max(h_1,h_2)$.  %
We need to sum (I) over $k(k-1)$ terms and divide by $\sigma_k^2$ as follows;
\begin{eqnarray*} \label{EAAEBB}
H_1=\frac{4}{\sigma_k^2}  \sum_{i \ne j}^k E(A_iA_j)E(B_iB_j)  &=
h^2 O \left( \frac{n\sum_{i \ne j} \xi_i^2\xi_j^2}{\vari}  \right)
\leq    h^2 O \left(  \frac{ n || \boldsymbol{\xi}||_2^4 }{\vari}  \right)
= O( \frac{h^2}{n})=o(1)
\end{eqnarray*}
from the condition in Theorem 1, $|| \boldsymbol{\xi} ||_2^4 = O(\frac{1}{n^2} \vari)$.


Additionally, we have
\begin{eqnarray*}
E(A_i) & = &  O(|\xi_i|)\\
E(B_iB_j^2) &=&  E \left[  \left( \frac{M_{1i}}{n_1} - \frac{M_{2i}}{n_2}   \right) \left( \frac{M_{1j}}{n_1} - \frac{M_{2j}}{n_2}   \right)^2  \right] \leq
 2 E \left[  \left( \frac{M_{1i}}{n_1} + \frac{M_{2i}}{n_2}   \right) \left( \frac{M_{1j}^2}{n_1^2} +  \frac{M_{2j}^2}{n_2^2}   \right)  \right]   \nonumber \\
&=& \underbrace{2  E \left[ \frac{M_{1i}M_{1j}^2}{n_1^3} \right]}_{ \textcircled{1} }
+   \underbrace{2E \left[\frac{M_{1i}M_{2j}^2}{n_1n_2^2} \right] }_{\textcircled{2} }
+    \underbrace{2 E\left[ \frac{M_{2i}M_{1j}^2}{n_2n_1^2} \right]}_{\textcircled{3} }
+    \underbrace{2E\left[ \frac{M_{2i}M_{2j}^2}{n_2^3} \right]}_{\textcircled{4}}
\end{eqnarray*}
which are
\begin{eqnarray*}
\textcircled{1} &=& h O \left( \frac{p_{1i}p_{1j}^2}{n^{1.5}} + \frac{p_{1i}p_{1j}}{n^2}  \right),~
\textcircled{2} = h O \left( \frac{p_{1i}p_{2j}^2}{n^{1.5}} + \frac{p_{1i}p_{2j}}{n^2} \right),\\
\textcircled{3} &=& h O \left( \frac{p_{2i}p_{1j}^2}{n^{1.5}} + \frac{p_{1i}p_{2j}}{n^2} \right),~
\textcircled{4} = h O \left( \frac{p_{2i}p_{2j}^2}{n^{1.5}} + \frac{p_{2i}p_{2j}}{n^2}  \right).
\end{eqnarray*}
Therefore we have
\begin{eqnarray*}
H_2 &=&\frac{4}{\sigma_k^2} \sum_{i \ne j} E(A_i)E(B_iB_j^2)\\
 &\le& h O \left( \frac{\sum_{i \ne j}   |\xi_i|
\left( \frac{p_{1i}p_{1j}^2}{n^{1.5}} + \frac{p_{1i}p_{1j}}{n^2} + \frac{p_{1i}p_{2j}^2}{n^{1.5}} + \frac{p_{1i}p_{2j}}{n^2}
+\frac{p_{2i}p_{1j}^2}{n^{1.5}}
+\frac{p_{2i}p_{2j}^2}{n^{1.5}}
+\frac{p_{2i}p_{2j}}{n^{2}}
\right)}
{ n^{-2} \left( \vari \right) }    \right) \nonumber \\
& =& h O \biggl( \frac{ \sqrt{n} \sum_{i \ne j} |\xi_i| p_{1i}p_{1j}^2 }{\vari}
+ \frac{ \sum_{i \ne j} |\xi_i| p_{1i}p_{1j}}{\vari} + \frac{\sqrt{n} \sum_{i \ne j} |\xi_i| p_{1i}p_{2}^2 }{\vari} \nonumber \\
&&+ \frac{ \sum_{i \ne j} |\xi_i| p_{1i}p_{2j} }{\vari }
+ \frac{\sqrt{n}\sum_{i \ne j} |\xi_i| p_{2i}p_{2j}^2}{\vari} +  \frac{\sum_{i \ne j} |\xi_i| p_{2i}p_{2j}}{\vari}  \biggr)  \nonumber \\
&=& h O \left(  \frac{\sqrt{n} \left( \abs{{\boldsymbol \xi}} \cdot {\bf P}_1 \right) || {\bf P}_1 ||_2^2}{\vari}  \right)
+  h O \left(  \frac{  \abs{{\boldsymbol \xi}} \cdot {\bf P}_1  }{\vari}  \right)
+ h O \left(  \frac{\sqrt{n} \left( \abs{{\boldsymbol \xi}} \cdot {\bf P}_1 \right) ||{\bf P}_2||_2^2  }{\vari}  \right)  \nonumber \\
&&+ h O \left(  \frac{ \abs{{\boldsymbol \xi}} \cdot {\bf P}_1 }{\vari}  \right)
+  h O \left(  \frac{  (\abs{{\boldsymbol \xi}} \cdot {\bf P}_2) ||{\bf P}_2 ||_2^2    }{\vari}  \right)
+ hO\left(\frac{|\boldsymbol \xi |\cdot {\bf P}_2}{\vari} \right) \nonumber \\
&=& hO\left(   \frac{ |\boldsymbol \xi|\cdot ( {\bf P}_1 + {\bf P}_2) (\psum)}{\vari}  \right)
+ hO\left( \frac{|\boldsymbol \xi|\cdot ({\bf P}_1 + {\bf P}_2)}{\vari}\right) \\
&=&O(h) +O(h) =O(h)
\end{eqnarray*}
{ where the second last equality is obtained from 5 in Lemma.S1 and   $||{\bf P}_1 + {\bf P}_2 ||_2^2 =O(1)$.  }

Lastly, $H_3$ is equivalent to the following:
\begin{align*}
H_3 &= \frac{1}{\sigma_k^2} \sum_{i \neq j}E \left[ \left( \frac{M_{1i}}{n_1} - \frac{M_{2i}}{n_2} \right)^2   \left( \frac{M_{1j}}{n_1} - \frac{M_{2j}}{n_2} \right)^2 \right] \nonumber \\
&\leq  \frac{4}{\sigma_k^2}\sum_{i \neq j} E \left[ \left( \frac{M_{1i}^2}{n_1^2}  + \frac{M_{2i}^2}{n_2^2}  \right)  \left( \frac{M_{1j}^2}{n_1^2} + \frac{M_{2j}^2}{n_2^2}  \right) \right] \nonumber \\
&\leq  \frac{4}{\sigma_k^2}\sum_{i \neq j} \left\{ E\left( \frac{M_{1i}^2M_{1j}^2}{n_1^4} \right) + E\left( \frac{M_{1i}^2M_{2j}^2}{n_1^4} \right) +
 E\left( \frac{M_{2i}^2M_{1j}^2}{n_1^4} \right) +   E\left( \frac{M_{2i}^2M_{2j}^2}{n_1^4} \right) \right\}   \\
&= \frac{h^4}{\sigma_k^2} O\left( \sum_{i \neq j} \frac{p_{1i}^2p_{2j}^2}{n} \right) + \frac{h^2}{\sigma_k^2}O\left( \sum_{i \neq j} \frac{p_{1i}p_{2j}}{n^3}  \right)
\end{align*}
from  $ E\left( \frac{M_{1i}^2M_{1j}^2}{n_1^4} \right) =   O(h^4\frac{p_{1i}^2p_{1j}^2}{n^2} + h^2\frac{p_{1i}p_{1j}}{n^3})$ and similar results for the other terms.
Next, to show that $ H_3 = \frac{1}{\sigma_k^2}\sum_{i \ne j}^k E(B_i^2B_j^2) \rightarrow 0$ as follows;
using $\sum_{i\neq j} p_{1i}^2 p_{2j}^2  \leq ||{\bf P}_1 ||_2^2 ||{\bf P}_2 ||_2^2 \leq || {\bf P}_1 ||_2 ||{\bf P}_2 ||_2
\leq ( ||{\bf P}_1 ||_2^2 + ||{\bf P}_2 ||_2^2) =O(\vari)$ from 2 in Lemma.S1 and $||{\bf P}_c ||_2 \leq 1$ for $c=1,2$, we have
\begin{align*}
H_3
 &=  h^4 O\left( \frac{\sum_{i \ne j}^k p_{1i}^2p_{1j}^2}{\vari} \right)
 +  h^2 O\left( \frac{\sum_{i \ne j}^k p_{1i}p_{1j} }{n \left( \vari \right)}  \right)\\
 &= h^4 O\left( \frac{ ||{\bf P}_1 ||_2^2 ||{\bf P}_2 ||_2^2   }{\vari } \right)
 + h^2 O\left(\frac{1}{n(\vari)} \right)
 = \underbrace{h^4 O(1)}_{(I)} + \underbrace{h^2 O(1)}_{(II)}=O(h^2)
\end{align*}
where the second term is obtained from  the condition 3 in Theorem 1.



\item  We show $ W_4 = \frac{1}{\sigma_k^2} E(\sum_{i\neq j} R_iR_j) \rightarrow 0$.
\begin{eqnarray*}
 && \sum_{i\neq j}   E \left[ \left( \kappa_{1i} M_{1i} + \kappa_{2i} M_{2i} \right) \left( \kappa_{1j}M_{1j} +\kappa_{2j} M_{2j} \right) \right] \nonumber \\
	&=& \sum_{i\neq j}\kappa_{1i}\kappa_{1j}E(M_{1i}M_{1j}) +  2 \sum_{i\neq j} \kappa_{1i}\kappa_{2j}E(M_{1i})E(M_{2j}) +\sum_{i \neq j} \kappa_{2i}\kappa_{2j} E(M_{2i}M_{2j}) \nonumber \\
	&=&  h O\left( n\sum_{i\neq j} \left(\frac{1}{n^4} + \frac{|\xi_{i}|+|\xi_{j}|}{n^3}
+ \frac{|\xi_i||\xi_j|}{n^2} \right)  p_{1i}p_{1j} \right)
+2h^2 O\left( n \sum_{i\neq j}\left(\frac{1}{n^4} + \frac{|\xi_{i}|+|\xi_{j}|}{n^3} + \frac{|\xi_i||\xi_j|}{n^2} \right) p_{1i}p_{2j} \right) \nonumber \\
&& +h  O\left( n \sum_{i\neq j} \left( \frac{1}{n^4} + \frac{|\xi_{i}|+|\xi_{j}|}{n^3} + \frac{|\xi_i||\xi_j|}{n^2} \right)p_{2i}p_{2j}  \right).
\end{eqnarray*}
Using  $\sum_{i\neq j} p_{1i} p_{2j} \leq \sum_{i=1}^k p_{1i} \sum_{j=1}^k p_{2j} =1$,
$n \left( {\bf P_1}\cdot {\bf P}_2+a_k \right) \geq \epsilon>0$ for some $\epsilon$ and
\begin{eqnarray*}
\sum_{i\neq j} (|\xi_i| +|\xi_j|) p_{1i}p_{2j} \leq \sum_{i=1}^k |\xi_i| p_{1i} \sum_{j =1}^k p_{2j}
+\sum_{j=1}^k |\xi_j| p_{2j} \sum_{i =1}^k p_{1j} \\
 \leq  \sum_{i=1}^k |\xi_i| p_{1i} + \sum_{j=1}^k |\xi_j| p_{2j} =  \sum_{i=1}^k |\xi_i|(p_{1i}+p_{2i}) =
 |\boldsymbol{\xi}|\cdot ({\bf P}_1 + {\bf P}_2),
 \end{eqnarray*}
 for $|{\boldsymbol \xi}| = (|\xi_1|,\ldots, |\xi_k|)$,
    we have
\begin{eqnarray*}
\frac{\sum_{i\neq j}E(R_i R_j)}{\sigma_k^2} &=& h^2 O\left( \frac{\sum_{i\neq j}p_{1i}p_{2j}}{n \vari  }\right)
+ h^2 O \left(\frac{\sum_{i\neq j} (|\xi_i| + |\xi_j|)p_{1i}p_{2j} }{ \vari}\right)\\
&&+ h^2 O \left(\frac{n\sum_{i\neq j}|\xi_i||\xi_j| p_{1i}p_{2j}}{\vari} \right) \nonumber \\
&=& h^2 O\left( \frac{1}{ \epsilon} \right)
+  h^2 O \left( \frac{  |{\boldsymbol \xi}|\cdot({\bf P}_1 + {\bf P}_2)}{ \vari }\right)
+  h^2 O\left(  \frac{n (|{\boldsymbol \xi}|\cdot {\bf P}_1)(|{\boldsymbol \xi}|\cdot {\bf P}_2) }{\vari }   \right).
\end{eqnarray*}
{From $ |{\boldsymbol \xi}|*({\bf P}_1 + {\bf P}_2) = O(\vari)$
and   $ n (|{\boldsymbol \xi}|\cdot {\bf P}_1)(|{\boldsymbol \xi}|\cdot {\bf P}_2)  =O(\vari)$ from 6 and 7
in Lemma.S1,
we obtain $\frac{\sum_{i\neq j}E(R_i R_j)}{\sigma_k^2}  = o(1)$ as $h = o(1)$.   }

\item  We show   $W_5 = \frac{1}{\sigma_k^2} E( \sum_{i\neq j}Q_iR_j) \rightarrow 0$.
We have
\begin{align*}
\frac{1}{\sigma_k^2}\sum_{i\neq j} E(Q_iR_j) = \frac{1}{\sigma_k^2}\sum _{i \neq j}E \left[ (2A_iB_i + B_i^2)R_j   \right] &=
\underbrace{\frac{2}{\sigma_k^2}\sum_{i\neq j}E(A_iB_iR_j)}_{K_1} +
\underbrace{\frac{1}{\sigma_k^2}\sum_{i \neq j} E(B_i^2R_j)}_{K_2}
\end{align*}
We'll  look at $K_1$ and $K_2$ separately.
Since $L_{1i}s$ and $M_{1i}s$ are independent, we have $K_1  = \frac{2}{\sigma_k^2}\sum_{i\neq j} E(A_i)E(B_iR_j)$.
For $E(A_i)$, we have
\begin{eqnarray}
E(A_i) &=& E \left( \frac{L_{1i}}{n_1} - \frac{L_{2i}}{n_2}  \right) = \frac{(n_1+u_1 n_1^{1/2})p_{1i}}{n_1} - \frac{n_2+u_2 n_2^{1/2}p_{2i}}{n_2}= O(|\xi_i|) \nonumber \label{eqn_EA}
\end{eqnarray}
Additionally, we have
\begin{eqnarray*}
|E(B_iR_j)|  &\leq & \Bigg| E \left[ \left( \frac{M_{1i}}{n_1} - \frac{M_{2i}}{n_2} \right)   \left( -\kappa_1 M_{1j} - \kappa_2 M_{2j}\right) \right] \Bigg| \nonumber \\
&\leq & E \left( \frac{|\kappa_1| M_{1i}M_{1j}}{n_1} + \frac{|\kappa_2| M_{1i}M_{2j}}{n_1}
+ \frac{|\kappa_1| M_{2i}M_{1j}}{n_2} + \frac{|\kappa_2| M_{2i}M_{2j}}{n_2}  \right) \nonumber \\
&=& h^2 O \left(  \frac{|\xi_j| p_{1i}p_{1j}}{n}  + \frac{|\xi_j| p_{1i}p_{2j}}{n} + \frac{|\xi_j| p_{2i}p_{1j}}{n} + \frac{|\xi_j| p_{2i}p_{2j}}{n} \right)
\end{eqnarray*}
As before, we need to sum over $k(k-1)$ terms and divide by $\sigma_k^2$ as follows;
\begin{eqnarray*}
 K_1 &=& h^2 O \left( \frac{n \sum_{i \ne j} \abs{\xi_i}\abs{\xi_j} p_{1i}p_{1j}}{\vari}  \right) +  h^2 O \left( \frac{n \sum_{i \ne j} \abs{\xi_i}\abs{\xi_j} p_{1i}p_{2j}}{\vari}  \right) \nonumber \\
&&+ h^2 O \left( \frac{n \sum_{i \ne j} \abs{\xi_i}\abs{\xi_j} p_{2i}p_{1j}}{\vari} \right) + h^2 O \left( \frac{n \sum_{i \ne j} \abs{\xi_i}\abs{\xi_j} p_{2i}p_{2j}}{\vari} \right) \nonumber \\
&=&   h^2 O \left( \frac{n \left( \abs{{\boldsymbol \xi}} \cdot {\bf P}_1 \right)^2 }{\vari} \right) + 2h^2 O \left( \frac{n_1 \left( \abs{{\boldsymbol \xi}} \cdot {\bf P}_1\right) \left( \abs{{\boldsymbol \xi}} \cdot {\bf P}_2   \right) }{\vari} \right) + h^2 O \left( \frac{n \left( \abs{{\boldsymbol \xi}} \cdot {\bf P}_2 \right)^2 }{ \vari  } \right) \nonumber \\
&=&  h^2 O \left( \frac{n \left( \abs{{\boldsymbol \xi}} \cdot {\bf P}_1 + \abs{{\boldsymbol \xi}} \cdot {\bf P}_2 \right)^2 }
{\vari} \right) = O(h^2) = o(1)
\end{eqnarray*}
from {$n(|\boldsymbol \xi|\cdot ({\bf P}_1 + {\bf P}_2) )^2 = O(\vari)$ from 7 in Lemma.S1 and $h \rightarrow 0$.}

For $K_2$,  we have
\begin{eqnarray*}
| K_2| &=& \frac{1}{\sigma_k^2} \Bigg|\sum_{i\neq j}E \left[ \left( \frac{M_{1i,}}{n_1} - \frac{M_{2i}}{n_2}   \right)^2  \left( \frac{\kappa_{1j}M_{1j}}{n_1^2} + \frac{\kappa_{2j}M_{2j}}{n_2^2}    \right)  \right] \Bigg|\\
 &\leq&\frac{1}{\sigma_k^2}
\sum_{i \neq j} 2 E \left[ \left( \frac{M_{1i}^2}{n_1^2} + \frac{M_{2i^2}}{n_2^2}   \right)
  \left( \frac{|\kappa_{1j}|M_{1j}}{n_1^2} + \frac{|\kappa_{2j}|M_{2j}}{n_2^2}    \right)  \right]    \nonumber  \\
	&=& \frac{1}{\sigma_k^2}  \sum_{i\neq j} \left( \frac{|\kappa_{1j}| E(M_{1i}^2M_{1j})}{n_1^4}
   + \frac{ |\kappa_{2j}|E(M_{1i}^2) E(M_{2j})}{n_1^2n_2^2}
  + \frac{|\kappa_{1j}| E (M_{1j} M_{2i}^2)}{n_1^2n_2^2}
 + \frac{|\kappa_{2j}| E (M_{1i}^2 M_{2j})}{n_2^4} \right)\\
&=& h^2O( \frac{1}{n^{3.5} \sigma_k^2 }) =o(1)
\end{eqnarray*}
using $ n^{3.5} \sigma_k^2 \asymp  \sqrt{n} n\vari  \geq \sqrt{n} \epsilon \rightarrow \infty $ from the condition 2 in Theorem 1.
%
Thus, we have shown that  $W_5 \rightarrow 0$.

\section{Proof of Lemma 4}
When $X_{1i}$ and $X_{2i}$, for $1\leq i \leq k$, come from independent Poisson distributions for all $1\leq i \leq k$,
from $ F_k \equiv  \frac{ \sumk f(X_{1i},X_{2i}) }{\sigma_k}
= \frac{\sumk {\cal G}_{1i}(X_{1i},X_{2i})}{\sigma_k }  + \frac{\sumk {\cal G}_{2i}(X_{1i},X_{2i})}{\sigma_k }$,
we show  $(i)$ $\frac{\sumk {\cal G}_{1i}(X_{1i},X_{2i})}{\sigma_k }  \cd N(0,1)$
and $(ii)$ $\frac{\sumk {\cal G}_{2i}(X_{1i},X_{2i})}{\sigma_k } \cp 0$.

The asymptotic normality of $ \frac{ \sum_{i=1}^k {\cal G}_{1i}(X_{1i},X_{2i})}{\sigma_k}$ is obtained from the Lyapounov's condition (Billingsley (1995)) as follows:
first,  we have
$ {\cal G}_{2i}(X_{1i},X_{2i})  =  (\frac{X_{1i}}{n_1} - \frac{X_{2i}}{n_2} )^2 -(\frac{X_{1i}}{n_1^2} + \frac{X_{2i}}{n_2^2}) -(p_{1i}-p_{2i})^2    =  (\frac{X_{1i}}{n_1} -p_{1i})^2 + (\frac{X_{2i}}{n_2} -p_{2i})^2 + 2 \xi_i (\frac{X_{1i}}{n_1} -p_{1i})
+ \xi_i (\frac{X_{2i}}{n_2} -p_{2i}) -(\frac{X_{1i}}{n_1^2} + \frac{X_{2i}}{n_2^2})$ where $\xi_i = p_{1i}-p_{2i}$
and we check the Lyapounov's condition which is
\begin{eqnarray*}
\frac{\sumk E \left[ {\cal G}_{2i}(X_{1i},X_{2i})^4 \right]}{\sigma_k^4}
&\leq&   \underbrace{\frac{1}{\sigma_k^4} \sumk \sum_{c=1}^2 E(\frac{X_{ci}}{n_c} -p_{ci})^8 }_{(I)}
 + \underbrace{\frac{1}{\sigma_k^4}\sumk \sum_{c=1}^2 \xi_i^4  E \left(\frac{X_{ci}}{n_c} -p_{ci} \right)^4}_{(II)}  \\
&& + \underbrace{\sumk E\left( \frac{X_{1i}}{n_1^2} + \frac{X_{2i}}{n_2}^2\right)^4}_{(III)}.
\end{eqnarray*}
Since we have  $E(Y-\lambda)^{2m} = O(\sum_{i=1}^m \lambda^i) =O( \lambda^m + \lambda)$
for $Y \sim poisson(\lambda)$,   we have
\begin{eqnarray}
(I) &\leq&  O\left( \frac{\sumk \sum_{c=1}^2 \left( p_{ci}^4   +  \frac{p_{ci}}{n_c^3} \right)}{ \var{4}  } \right)
=    O\left( \frac{\sum_{c=1}^2 \max_{i} p_{ci}^2}{\vari} \right)  +             O\left( \frac{1}{n^3 \var{4}} \right)\\
 &\leq &  O\left( \frac{\sum_{c=1}^2 \max_{i} p_{ci}^2}{||{\bf P}_c ||_2^2} \right)  + O\left( \frac{1}{n^3 \var{4}} \right)   =     o(1) \nonumber
\end{eqnarray}
where the first $O(\cdot)$ term is $o(1)$ due to the condition 2 in Theorem 1
and the second $O(\cdot)$ term is also $o(1)$ due to the condition 3 in Theorem 1.
Similarly, from $E(Y^4) = O(\sum_{m=1}^4 \lambda^m) =O(\lambda^4 + \lambda)$
for $Y\sim poisson(\lambda)$, we have
\begin{eqnarray*}
(II) &\leq&  O\left( \frac{n^2 \sumk \xi_i^4 (p_{ci}^2 + \frac{p_{ci}}{n_{ci}})}{\var{4}}  \right)
=    \frac{ \max_i p_{1i}^2 + \max_i p_{2i}^2 }{\vari}  O\left(  \frac{ n^2 ||\boldsymbol{\xi} ||_2^4  }{\vari}\right)
+  O\left(\frac{1}{n^3 \var{4}} \right) \\
 &=& o(1) O(1) +o(1) =o(1)
 \end{eqnarray*}
 where the first  $O(\cdot)$ term is $o(1)$ due to
 the the result 2 in Lemma.S1 and the condition 4 in Theorem 1
 and the second  $O(\cdot)$ term is $o(1)$  due to the condition 3 in Theorem 1.
Lastly, we have
\begin{eqnarray*}
(III) &=&  O\left( \frac{\sumk \sum_{c=1}^2\left( p_{ci}^4 + \frac{p_{ci}}{n_c^3} \right) }{\var{4} }   \right)  \\
&=&   O\left(\sum_{c=1}^2 \max_i p_{ci}^2 \right)  +           O\left( \frac{1}{n^3 \var{4}} \right)  =o(1)
\end{eqnarray*}
from  the result 1 in Lemma.S1 in the Supplementary material and the condition 3 in Theorem 1.
Combining these results, we prove the Lyapounov's condition is satisfied, so we have
the asymptotic normality of $\frac{\sumk {\cal G}_{1i}(X_{1i},X_{2i})}{\sigma_k }  \cd N(0,1)$.

For $(ii)$ $\frac{\sumk {\cal G}_{2i}(X_{1i},X_{2i})}{\sigma_k } \cp 0$,
we see that $E\sumk {\cal G}_{2i}(X_{1i},X_{2i})=0$, so it is sufficient to show
$Var(\frac{\sumk {\cal G}_{2i}(X_{1i},X_{2i})}{\sigma_k } ) \rightarrow 0$.
This can be shown by
\begin{eqnarray*}
Var(\frac{\sumk {\cal G}_{2i}(X_{1i},X_{2i})}{\sigma_k } )
&=& O\left( \frac{  n \sumk \xi_i^2 (p_{1i} + p_{2i})    }{   \vari }\right)  \\
&=&  (\max_i p_{1i} + \max_i p_{2i}) O\left( \frac{n|| \boldsymbol{\xi}||_2^2  }{\vari} \right)=o(1)
\end{eqnarray*}
from the $\max_i p_{ci} =o(1)$ in the result 2 in Lemma.S1 in the Supplementary material and
the condition 4 in Theorem 1.
Using $(i)$ and $(ii)$, we conclude
$F_k   \cd N(0,1)$.

\end{enumerate}
\qedSolid
\end{document}